\documentclass[11pt,letterpaper]{amsart}%
\usepackage{amsmath}
\usepackage{txfonts}
\usepackage{graphicx}
\usepackage{amsthm}
\usepackage{amsfonts}
\usepackage{amssymb}
\usepackage[utf8]{inputenc}
\usepackage{comment}
\usepackage{xcolor}
\usepackage{lipsum}
\usepackage{ifthen}%
\usepackage{tikz}

\usepackage{newtx}
\usepackage{relsize}

\providecommand{\U}[1]{\protect\rule{.1in}{.1in}}
\numberwithin{equation}{section}
\newtheorem{theorem}{Theorem}[section]
\newtheorem{lem}[theorem]{Lemma}
\newtheorem{thm}[theorem]{Theorem}
\newtheorem{pro}[theorem]{Proposition}
\newtheorem{cor}[theorem]{Corollary}
\newtheorem{defi}[theorem]{Definition}
\newtheorem{rem}[theorem]{Remark}

\def\Area{\operatorname{Area}\,}
\def\Re{\operatorname{Re}}
\def\Res{\operatorname{Res}}
\def\euc{\operatorname{euc}}
\def\loc{\operatorname{loc}}

\def\S2{\mathbb{S}^2}
\def\s{\,\,\,\,}

\def\endproof{$\hfill\Box$}
\def\R{\mathbb{R}}

\def\S{\mathcal{S}}
\def\M{\mathcal{M}}
\def\K{\mathbb{K}}

\def\diam{\mathrm{diam}\,}
\def\div{\mathrm{div}}

\allowdisplaybreaks
\def\s{\,\,\,\,}

\def\endproof{$\hfill\Box$\\}
\def\R{\mathbb{R}}

\def\S{\mathcal{S}}
\def\M{\mathcal{M}}
\def\K{\mathbb{K}}
\def\res{\mathrm{Res}}

\allowdisplaybreaks

\newcommand{\chapter}{\part}

\title[Prescribed curvature with singularities]{Prescribing negative curvature with cusps and conical singularities on compact surface}

\author{Jingyi Chen, Yuxiang Li, Yunqing Wu }
\address{Department of Mathematics\\
	The University of British Columbia, Vancouver, Canada}
\email{jychen@math.ubc.ca}
\address{
	Department of Mathematics\\
	Tsinghua University, Beijing, China}
\email{liyuxiang@mail.tsinghua.edu.cn}
\address{
	The Institute of Geometry and Physics\\
	University of Science and Technology of China, Hefei, Anhui, People's Republic of China}
\email{yqwu19@ustc.edu.cn}

\thanks{Chen is partially supported by NSERC Discovery Grant GR010074}
\thanks{Li is partially supported by  National Key R\&D Program of China 2022YFA1005400}

\allowdisplaybreaks[4]

\date{}

\begin{document}
	
	\maketitle

	\vspace{-.2in}
	
	\begin{abstract}
		On a compact surface, we prove existence and uniqueness of the conformal metric whose curvature is prescribed by a negative function away from finitely many points where the metric has prescribed angles presenting cusps or conical singularities.
	\end{abstract}
	
	\tableofcontents
	%\clearpage
	
	%\part{}	
	
	\section{Introduction}
	
	For a compact surface $\Sigma$ with a Riemannian metric $g_{0}$ with 
	Gauss curvature $K_{g_0}$ the classical Nirenberg's question asks whether there is a Riemannian metric $g$ such that $g=e^{2u}g_0$ and its Gauss curvature equals a given function $K$. This is equivalent to ask whether 
	\begin{align}
	\label{regular conformal metric equation}
	\Delta_{g_{0}} u=  K_{g_{0}}-K e^{2u}
	\end{align}
	is solvable, under the necessary constraints imposed on $K$ in terms of the Euler characteristic 
	$\chi(\Sigma)$ via the Gauss-Bonnet theorem. When $\chi(\Sigma)<0$, Berger showed in \cite{Berger} that \eqref{regular conformal metric equation} is uniquely solvable if $K<0$. When $\chi(\Sigma)=0$, Kazdan-Warner \cite{Kazdan-W} proved that \eqref{regular conformal metric equation} is solvable 
	if and only if $K\equiv 0$ or $K$ changes sign and $\int_{\Sigma} K e^{2v}d V_{g_{0}}<0$ where $\Delta_{g_0} v = K_{g_0}$. When $\chi(\Sigma)>0$, Chang-Yang \cite{chang1987prescribing, chang1988conformal} and Chen-Ding \cite{chen1987scalar} found sufficient conditions for solvability. 
	
	When $g$ has prescribed singular behaviour at a finite set of points on $\Sigma$, the prescribing curvature problem 
	has been set up and studied. A real divisor $\mathfrak{D}$ on $\Sigma$ is a formal sum
	\begin{align*}
	\mathfrak{D}=\sum_{i=1}^{m} \beta_{i} p_{i}, 
	\end{align*}
	where $\beta:=(\beta_1,\cdots,\beta_m)\in\R^m$ and $p_{i} \in \Sigma$ are distinct. The Euler characteristic of the pair $(\Sigma,\beta)$ is defined to be
	\begin{align*}
	\chi(\Sigma,\beta):=\chi(\Sigma)+\sum_{i=1}^{m} \beta_{i}. 
	\end{align*}
	A conformal metric $g$ on $(\Sigma,g_{0})$ is said to represent the divisor $\mathfrak{D}$ if $g$ is a smooth metric away from $p_{1},\cdots,p_{m}$ such that each $p_{i}$ has a neighbourhood $U_{i}$ with a coordinate $z_{i}$ such that $z_{i}(p_{i})=0$ and $g$ is in the form
	\begin{align*}
	g=e^{2v} |z_{i}|^{2 \beta_{i}}g_{\euc},
	\end{align*}
	where $v \in C^0(U_{i}) \cap C^{2}(U_{i} \backslash \{p_{i}\})$. 
	The point $p_{i}$ is called a conical singularity of angle $\theta_i=2\pi (\beta_{i}+1)$ if $\beta_{i}>-1$ and a cusp if $\beta_i=-1$. 
	
	Finding $g=e^{2u}g_{0}$ represent $\mathfrak{D}$ with a prescribed function $K$ is equivalent to solve
	\begin{align}
	\label{singular conformal metric equation}
	-\Delta_{g_{0}} u= K e^{2u} -K_{g_{0}}-2 \pi \sum_{i=1}^{m} \beta_{i} \delta_{p_{i}}
	\end{align}
	in the sense of distributions, where $\delta_{p_{i}}$ is the Dirac measure at $p_{i}$.
	
	In \cite{McOwen}, McOwen proved that for a compact Riemann surface  $(\Sigma,g_{0})$ with Gauss curvature $K \equiv -1$ and a divisor $\mathfrak{D}$, if $\chi(\Sigma,\beta)<0$ and all $\beta_{i}>-1$, then there exists a unique conformal metric $g$ representing $\mathfrak{D}$ with curvature $-1$ (also see the reference therein for earlier history). Troyanov further proved a more general result for nonpositive variable curvature with prescribed conical singularities:  
	\begin{thm}\cite[Theorem 1]{Troyanov1991}
		\label{thm Troyanov negative curvature existence}
		Let $\mathfrak{D}=\sum_{i=1}^{m} \beta_{i} p_{i}$ be a divisor on a compact surface $\Sigma$ with $\beta_{1},...,\beta_m>-1$. Let $\displaystyle{K}<0$ be a H\"{o}lder continuous function on $\Sigma$. If $\chi(\Sigma,\beta)<0$, then there exists a unique conformal metric $g$ on $\Sigma$ representing $\mathfrak{D}$ with curvature $K$. 
	\end{thm}
	Variational methods have been successful for studying Nirenberg's problem (with or without $\mathfrak{D}$), where the Moser-Trudinger inequality is used to guarantee the coercivity which asserts that the solution is a minimizer of some energy functional. However, this method fails if $\beta_{i}=-1$, i.e. some $p_i$ are cusps. In \cite[Theorem B]{HT}, existence and uniqueness of a conformal metric $g$ representing $\mathfrak{D}$ with curvature $K$ for $\chi(\Sigma,\beta)<0$ and $0\not\equiv K\leq 0$ is obtained under a pinching condition: $aK_1\leq K\leq bK_1$ on $\Sigma\backslash\{p_1,...,p_m\}$ where $K_1$ is the curvature of a conformal metric $g_1$ representing 
	$\mathfrak{D}$ and $a,b$ are  positive constants.

	\vspace{.1cm}
	
	In this work we generalize Theorem \ref{thm Troyanov negative curvature existence} by allowing cusp singularities and weaker regularity on $K$ and the aforementioned result by removing the pinching condition. 
	%to compact surface with cuspidal divisors and negative curvature: 
	\begin{thm}
		\label{thm existence of metrics with cusps}
		Let $\mathfrak{D}=\sum_{i=1}^{m} \beta_ip_i$ be a divisor on a closed surface $(\Sigma,g_0)$ with
		$\chi(\Sigma,\beta)<0$ and $\beta_{1},...,\beta_m \geq -1$. 
		Then for any $K\in L^\infty(\Sigma)$ with $K\leq -1$ there exists a unique metric $g=e^{2u}g_0$   
		representing  $\mathfrak{D}$ with curvature $K$. Moreover $g$ is complete and $\Area(\Sigma,g)$ is finite.
	\end{thm}

	 Rather than variational methods,  the existence assertion in Theorem \ref{thm existence of metrics with cusps} is proved via an approximation argument: choosing a sequence of divisors $\mathfrak{D}_{k}=\sum_{i=1}^{m} \beta_{i}^{k} p_{i}$ with $-1<\beta_{i}^{k} \rightarrow \beta_{i}$ and a sequence of curvature functions $K_{k}$ approximating $K$, we show the metric $g_{k}=e^{2u_{k}}g_{0}$ representing $\mathfrak{D}_{k}$ as in Theorem \ref{thm Troyanov negative curvature existence} converges to the desired metric. In order to deduce the convergence, an essential ingredient is an {\it area identity} (Theorem \ref{theorem area convergence of bubble tree}). The identity is established for the case that Gauss curvatures are bounded away from zero by a constant, except at finitely many points where curvature concentrates as $\delta$-measure. The next key observation is the absence of bubbles for the {\it negative curvature} case, therefore the convergence follows.  
		
		%We apply this area conservation to obtain convergence when we use Troyanov's solutions in Theorem \ref{thm Troyanov negative curvature existence} to approximate cusps. 

		Area identities also play important roles in the case of positive curvature, where a lot of work has been done on convergence of conformal metrics with positive curvature (possibly representing divisors), for instance, in  \cite{Brezis1991}, \cite{li-Shafrir1994}, \cite{li1999harnack}, \cite{bartolucci2002}, \cite{Bartolucci2007}, \cite{lin2010profile}, \cite{malchiodi2017variational} and the references therein. 
		%A key ingredient to analysing the convergence process is the area identity, which can be used to describe the concentration phenomenon arising from a blow up analysis. For instance,  in \cite{li-Shafrir1994}, Li and Shafrir established an area identity for bounded nonnegative continuous Gauss curvature near a point whose Gauss curvature is positive. 

		The area identity fails when curvatures change sign (see Remark \ref{K cannot change sign}). Volume convergence in the Gromov-Hausdorff limits has been obtained for compact manifolds with a uniform lower bound on Ricci curvature \cite{Colding}, and for Alexandrov spaces with curvature bounded above uniformly with respect to the Gromov-Hausdorff distance  \cite{Ya}. 
		
		In a companion paper \cite{CLWpositive}, we focus on a {\it nonexistence} problem of prescribed positive curvature on $\mathbb S^2$ with prescribed conical singularities in which the area identity is applied.

	\medskip
	
	\noindent{\bf Acknowledgement.} The authors are grateful for the wonderful facility provided by Tsinghua University during their collaboration. A large portion of this work was carried out while the first author was visiting BICMR in 2023 and 2024. He would like to thank Professor Gang Tian for the invitation and for the hospitality of the institute.
	
	\section{Preliminary}
	In this section, we list some foundational tools and concepts which will be used in this paper, mostly from \cite{C-L1,C-L2}.

	\subsection{Curvature measure}
	%\subsection{Curvature as signed Radon measure} 
	Let $\Sigma$ be a closed surface with a Riemannian metric $g_0$, the Gauss curvature $K_{g_0}$ and the area element $dV_{g_0}$. For every $u$ in $L^1_{\loc}(\Sigma,g_0)$ we associate a Radon measure $g_u=e^{2u}g_0$ on $\Sigma$, and let $\M(\Sigma,g_0)$ be the set of $g_u$ so that there is a {\it signed} Radon measure $\mu(g_u)$ satisfying 
	\begin{equation*}
	\int_\Sigma\varphi \,d\mu(g_u)
	=\int_\Sigma \left(\varphi \,K_{g_0} - u \Delta_{g_0}\varphi\right)dV_{g_0},\s \forall \varphi\in C^\infty_0(\Sigma).
	\end{equation*}
	We write $dV_{g_u}= e^{2u} dV_{g_0}$
	% Now we introduce notations 
	%\begin{equation*}\label{notation}
	%dV_{g_u}:= e^{2u} dV_{g_0} \end{equation*}
	and call the signed Radon measure $\K_{g_u}= \mu(g_u)$ the {\it Gauss curvature measure}   (curvature measure for simplicity) for the Radon measure $g_u$. For smooth $u$ it is well-known
	$$
	K_{g_u} = e^{-2u}\left(K_{g_0}-\Delta_{g_0}u\right).
	$$
	
	In an isothermal coordinate chart $(x,y)$ for the smooth metric $g_0$, we can write $g_0= e^{2u_0}g_{\euc}$ for some local function $u_0$. So any $g\in\M(\Sigma,g_0)$ is locally expressible as $g=e^{2v}g_{\euc}$, where $v\in L^{1}_{\loc}(\Sigma)$
	and
	\begin{equation}
	-\Delta v \ dxdy =\K_{g_v}
	\end{equation}
	as distributions where $\Delta=\frac{\partial^2}{\partial x^2}+\frac{\partial^2}{\partial y^2}$.
	
	\begin{defi}\label{def represent divisors in distributions}\textnormal{
			Let $(\Sigma,g_0)$ be a surface and $\mathfrak{D}=\sum_{i=1}^m \beta_{i} p_{i}$ be a divisor on $\Sigma$. A Radon measure $g=e^{2u}g_0\in\mathcal{M}(\Sigma,g_0)$ {\it represents $\mathfrak{D}$ with curvature function $K$} if
			$$
			\K_{g}=K e^{2u} dV_{g_{0}} -2\pi\sum_{i=1}^{m} \beta_i\delta_{p_i},\s and\s K e^{2u}\in L^1(\Sigma,g_{0}).
			$$
			In other words, in the sense of distributions, 
			\begin{equation}\label{Gauss.equation}
			-\Delta_{g_0}u=K e^{2u}-K_{g_0}-2\pi\sum_{i=1}^{m} \beta_i\delta_{p_i}.
			\end{equation}}
	\end{defi}
	
	By \eqref{Gauss.equation}, for a closed surface $\Sigma$ we have 
	$$
	\frac{1}{2\pi}\int_{\Sigma}Ke^{2u} dV_{g_{0}}=\chi(\Sigma)+\sum_{i=1}^{m}  \beta_i=\chi(\Sigma,\beta).
	$$
	
	For simplicity, we use the notations: $\mathcal{M}(\Omega)=\mathcal{M}( \Omega,g_{\euc}  )$, where $\Omega \subset \mathbb{R}^2$ is a domain and 
	$\mathcal{M}(\mathbb{S}^{2})=\mathcal{M}(\mathbb{S}^{2}, g_{\mathbb{S}^{2}})$, where $g_{\mathbb{S}^{2}}$ is the metric on $\mathbb{S}^{2}$ of curvature 1.  
	
	\vspace{.1cm}
	
	We will use the following estimates from \cite{C-L2} for our later discussion. 
	
	\begin{comment}
	\begin{lem}\cite[Corollary 2.4, Lemma 2.5]{C-L2}
	\label{lemma estimate of solution solving a Radon measure}
	Let $\mu$ be a signed Radon measure on $D$. Suppose that $-\Delta u=\mu$ holds weakly and $\|u\|_{L^{1}(D)}<\gamma$. Then
	\begin{itemize}
	\item[(1)] $u \in W^{1,q}( D_{{1}/{2}}  )$ for any $q \in [1,2)$. Moreover, 
	\begin{align*}
	r^{1-\frac{2}{q}}\|\nabla u\|_{L^{q}(D_r(x))} \leq C(q) (\|u\|_{L^{1}(D)}+|\mu|(D)),\s \forall D_r(x)\subset D_{{1}/{2}};
	\end{align*}
	
	\item[(2)] if $|\mu|(D)<\tau<2\pi$, then for any $ p < \frac{ 4 \pi}{  \tau  }$, there is $\beta=\beta(\tau,\gamma,p)$ such that
	\begin{align*}
	\int_{ D_{{1}/{2}}} e^{p|u|} \leq \beta. 
	\end{align*}
	\end{itemize}
	\end{lem}
	\end{comment}
	
	\begin{thm}\label{thm convergence of solutions of Radon measure when measure is small}
		Let $\{ \mu_k \}$ be a sequence of signed Radon measures on $D$ with $|\mu_k|(D)<\epsilon_0< \pi$. Suppose that $-\Delta u_k=\mu_k$ holds weakly with $\|\nabla u_k\|_{L^1(D)}<\Lambda$ and $\Area(D,g_k)<\Lambda'$. Then after passing to a subsequence, one of the following holds
		\begin{itemize}
			\item[(1)]  $u_k\to u$ weakly in $W^{1,p}(D_{{1}/{2}})$ for any $p \in [1,2)$ and  $e^{2 u_k}\to e^{2 u}$ in $L^q(D_{{1}/{2}})$ for some $q>1$;
			\item[(2)] $u_k\to -\infty$ for a.e. $x$ and $e^{2u_k}\to 0$ in $L^q(D_{1/2})$ for some $q>1$.
		\end{itemize}
	\end{thm}
	
	\begin{comment}
	\begin{proof}
	Let $c_{k}$ be the mean value of $u_{k}$ on $D_{{1}/{2}}$. By the Poincaré inequality and Sobolev inequality, for any $q \in [1,2)$, $\|u_{k}-c_{k}\|_{ L^{q}(D_{{1}/{2}}  ) } \leq C$. Combining with Lemma \ref{lemma estimate of solution solving a Radon measure}, we know that $u_{k}-x_{k}$ is a bounded sequence in $\cap_{p \in [1,2)} W^{1,p}(D_{{1}/{2}})$. After passing to a subsequence, we may assume $u_{k}-c_{k}$ converges weakly to some function $u_{\infty}$ in $\cap_{p \in [1,2)} W^{1,p}(D_{{1}/{2}})$. 
	
	We choose $p_{0} \in (2,\frac{4 \pi }{\epsilon_{0}})$, then by Lemma \ref{lemma estimate of solution solving a Radon measure},
	\begin{align*}
	\int_{D_{{1}/{2}}} e^{p_{0} |u_{k}-c_{k}| } \leq \beta(\epsilon_{0},\Lambda,p_{0}),
	\end{align*}
	then for $q \in [1,\frac{p_{0}}{2})$, $e^{2(u_{k}-c_{k}) }$ converges to $e^{2u_{\infty}}$ in $L^{q}(D_{{1}/{2}})$. 
	
	Since $\Area(D,g_k)<\Lambda'$, then by Jensen's inequality, after passing to a subsequence, either $c_{k} \rightarrow c_{\infty}>-\infty$ or $c_{k} \rightarrow -\infty$. If $c_{k}$ converges to $c_{\infty}$, we set $u=u_{\infty}+c_{\infty}$, then $e^{2u_{k}}$ converges to $e^{2u}$ in $L^{q}(D_{{1}/{2}}  )$; if $c_{k} \rightarrow -\infty$, then $e^{2u_{k}}$ converges to $0$ in $L^{q}(D_{{1}/{2}}  ).$
	\end{proof}
	\end{comment}
	
	\begin{lem}
		\label{lemma gradient estimate on Riemann surfaces}
		Let $\mu$ be a signed Radon measure defined on a closed surface $(\Sigma,g_{0})$ and $u \in L^{1}(\Sigma,g_{0})$ solves $-\Delta_{g_{0}} u= \mu$ weakly. Then, for any $r>0$ and $q \in [1,2)$, there exists $C=C(q,r,g_{0})$ such that 
		\begin{align*}
		r^{q-2} \int_{ B_{r}^{g_{0}}(x)} |\nabla_{g_{0}} u |^{q} \leq C (|\mu|(\Sigma) )^{q}. 
		\end{align*}
	\end{lem}
	
	\subsection{Gauss-Bonnet formula}
	Let $g = e^{2u}g_{\euc} \in \mathcal{M}(D)$. There is a set $E\subset (0,1)$ such that the 1-dimensional Lebesgue measure of $(0,1)\backslash E$ is zero  (\cite{C-L2})  and $\int_{\partial D_t} \frac{\partial u}{\partial r}$ is well-defined for all $t \in E$. Furthermore, 
	\begin{equation}\label{Gauss-Bonnet}
	\K_{g}(D_t\backslash D_s) = -\int_{\partial D_t}\frac{\partial u}{\partial r} + \int_{\partial D_s}\frac{\partial u}{\partial r},
	\end{equation}
	\begin{equation}\label{Gauss-Bonnet.disk}
	\K_{g}(D_t) = -\int_{\partial D_t}\frac{\partial u}{\partial r}
	\end{equation}
	for all $t, s \in E$.
	The latter equation notably implies the convergence:
	\begin{equation}\label{Gauss-Bonnet.at.singularity1}
	\lim_{t \to 0, t \in E} \int_{\partial D_t} \frac{\partial u}{\partial r} = -\K_g(\{0\}).
	\end{equation}
	The set $\{t : \K_{g}(\partial D_t) \neq 0\}$ is countable, allowing to pick $E$ with
	$$
	\K_{g}(\partial D_t) = 0,\quad \text{for all } t \in E.
	$$
	
	Let $g_k = e^{2u_k}g_{\euc} \in \mathcal{M}(D)$ and assume that $u_k\to u$ weakly. Set $g = e^{2u}g_{\euc}$. We select a subset $\mathcal{A} \subset (0,1)$ with the following properties:
	\begin{enumerate}
		\item $(0,1) \backslash\mathcal{A}$ has zero (1 dimensional Lebesgue) measure;
		\item $\K_{g_k}(\partial D_t) = \K_g(\partial D_t) = 0$ for all $t \in \mathcal{A}$ and for all $k$;
		\item \eqref{Gauss-Bonnet.disk} holds for any $t \in \mathcal{A}$.
	\end{enumerate}
	
	From \cite[Theorem 1.40]{EG}, $\K_{g_k}(D_t) \rightarrow \K_{g}(D_t)$ for all $t \in \mathcal{A}$, in turn 
	\begin{equation}\label{convergence.u_r}
	\int_{\partial D_t}\frac{\partial u_k}{\partial r} \rightarrow \int_{\partial D_t}\frac{\partial u}{\partial r},\quad \text{for all } t \in \mathcal{A}.
	\end{equation}
	
	Now, we assume $\K_{g}(\partial D_t) = 0$ for any $t$. As $t\to s$, we observe
	$$
	\K_{g}(D_t\backslash D_s) \rightarrow 0.
	$$
	Let $\Phi_u(t) = \int_{\partial D_t}\frac{\partial u}{\partial r}, t \in \mathcal{A}$. For any $t_k \in \mathcal{A}$ with $t_k \rightarrow t_0$, $\Phi(t_k)$ is a Cauchy sequence. Consequently, we can extend the domain of $\Phi$ from $\mathcal{A}$ to $(0,1)$ by setting
	\begin{align}
	\label{continuous extension of phi}
	\Phi_u(t) = \lim_{s \in \mathcal{A}, s \rightarrow t}\Phi_u(s).
	\end{align}
	This extension ensures $\Phi_u$ is a  continuous function on $(0,1)$, and we establish the relationships:
	$$
	\K_{g}(D_t\backslash D_s) = -\Phi_u(t) + \Phi_u(s),\quad \K_{g}(D_t) = -\Phi_u(t),\quad \text{for all } s, t \in (0,1).
	$$

	\subsection{Bubble trees}\label{section.bubble.tree}
	In this subsection, we will adopt certain notations relevant to bubble trees from \cite{C-L1}.
	
	Let $g_k=e^{2u_k}g_{\euc}\in\mathcal{M}(D)$ with $|\K_{g_k}|(D)<\Lambda$. 
	
	\begin{defi}
		\label{def blowup sequence and bubble}\textnormal{
			A sequence $\{ (x_k,r_k) \}_{k=1}^{\infty}$ is a {\it blowup sequence of $\{u_{k}\}_{k=1}^{\infty}$ at $0$} if $x_k\rightarrow 0$, $r_k\rightarrow 0$ and $u_k'=u_k(x_k+r_kx)+\log r_k$ converges weakly to a function $u$ in $\cap_{p\in[1,2)}W^{1,p}_{\loc}(\R^2)$, and $u$ is called a {\it bubble}. The sequence $\{u_{k}\}^\infty_{k=1}$ {\it has no bubble at $0$} if no subsequence of $\{u_{k}\}^\infty_{k=1}$ has a blowup sequence at $0$. }
	\end{defi}
	
	The blowup is designated for invariance of area:
	$$
	\Area(D_R, e^{2u'_k} g_{\euc})=\Area(D_{Rr_k}(x_k),e^{2u_k} g_{\euc}).
	$$ 	
	We say $u_k'$ converges to a {\it ghost bubble} if there exists $c_{k} \rightarrow -\infty$ such that $u_k'-c_{k}$ converges weakly to some function $v$ in $\cap_{p\in[1,2)}W^{1,p}_{\loc}(\R^2)$. 
	
	The ghost bubble has the following property:
	\begin{align*}
	\lim_{r \rightarrow 0} \lim_{k \rightarrow \infty} \int_{ D_{{1}/{r} } \backslash \cup_{z \in \mathcal{S} } D_{r}(z)} e^{ 2u_k'  } =0,
	\end{align*}
	where $\mathcal{S}$ is the set consisting of measure-concentration points:
	$$
	\S= \left\{y: \mu(y) \geq \frac{\epsilon_{0}}{2}     \right\}.
	$$
	where $\mu$ is the weak limit of $|\mathbb{K}_{ e^{ 2 (u_k'-c_k) } g_{\euc} }|$ and $\epsilon_{0}$ is chosen as in Theorem \ref{thm convergence of solutions of Radon measure when measure is small}.
	
	In the sequel, we will simply write $\{ (x_k,r_k) \}, \{u_{k}\}$ instead of $\{(x_k,r_k)\}^\infty_{k=1}, \{u_{k}\}_{k=1}^{\infty}$ respectively. When $\{u_k\}$ has several blowup sequences at a point, we classify them according to the following definitions.
	
	\begin{defi}
		\label{def bubble different}
		\textnormal{Two blowup sequences $\{(x_k,r_k)\}$ and $\{(x_k',r_k')\}$ of $\{u_k\}$ at $0$ are {\it essentially different} if one of the following happens
			$$
			\frac{r_k}{r_k'}\rightarrow\infty,\mbox{ or }\,
			\frac{r_k'}{r_k}\rightarrow\infty,
			\s \mbox{or} \s\frac{|x_k-x_k'|}{r_k+r_k'}
			\rightarrow\infty.
			$$
			Otherwise, they are {\it essentially same}. }
	\end{defi}
	
	\begin{defi}
		\textnormal{We say {\it the sequence  $\{u_k\}$ has $m$ bubbles} if $\{u_k\}$ has $m$ essentially different blowup sequences and no subsequence of $\{u_k\}$ has more than $m$ essentially different blowup sequences.}
	\end{defi}
	
	\begin{defi}
		\label{def bubble on top}
		\textnormal{For two essentially different blowup sequences $\{(x_k,r_k)\}, \{(x_k',r_k')\}$, we say $\{ (x_k',r_k') \}$ is {\it on the top of $\{ (x_k,r_k) \}$} and write $\{ (x_k',r_k')\}<\{ (x_k,r_k)\}$, if
			$\frac{r_k'}{r_k}\rightarrow 0$ and
			$\frac{x_k'-x_k}{r_k}$ converges as $k\to\infty$. }
	\end{defi}
	
	\begin{rem}\textnormal{
			\label{remark different bubble relationship}
			(1) If $\{ (x_k',r_k') \}<\{ (x_k,r_k) \}$, we have
			$$
			l_k:=\frac{r_k'}{r_k}\rightarrow 0, \s
			y_{k}:=\frac{x_k'-x_k}{r_k}\rightarrow y.
			$$
			If we set $v_k=u_k(x_k+r_kx)+\log r_k$ and $v_k'=u_k(x'_k+r_k'x)+\log r_k'$, then we have
			\begin{align*}
			v_k'(x)&=u_k(x_k'+r_k' x)+\log r_k'=u_k\Big(x_k+r_k
			\Big(\frac{r_k'}{r_k}x+\frac{x_k'-x_k}{r_k}\Big)\Big)+\log r_k+\log 
			\frac{r_k'}{r_k}\\
			&=u_{k}( x_{k}+r_{k}(  l_{k}x+y_{k}   )   ) +\log r_{k}+\log l_{k}=    v_k \left(l_kx+y_k\right)+\log l_{k}. 
			\end{align*}
			Then $\{ (y_k,l_k) \}$ is a blowup sequence of $\{ v_k \}$ at $y$, and the limit of $v_k'$ can be considered as a bubble of $\{  v_k \}$.\vspace{.1cm}
			\newline 
			(2) If two essentially different blowup sequences $\{ (x_k,r_k)\}, \{ (x_k',r_k') \}$ are not on the top of each other, then we must have
			$$
			\frac{|x_k-x_k'|}{r_k+r_k'} \rightarrow \infty,
			$$
			which yields that for any $t>0$, 
			$$
			D_{tr_k}(x_k)\cap D_{tr_k'}(x_k')=\emptyset
			$$
			when $k$ is sufficiently large.}
	\end{rem}
	
	\begin{defi}
		\label{def bubble right on top}
		\textnormal{We say  $\{ (x_k',r_k') \}$ is {\it right on the top} of $\{ (x_k,r_k) \}$, if $\{ (x_k',r_k') \}<\{ (x_k,r_k) \}$ and there is no blowup sequence $\{ (x_k'',r_k'') \}$, such that
			$$
			\{ (x_k',r_k') \}<\{ (x_k'',r_k'') \}<\{ (x_k,r_k) \}.
			$$
		}
	\end{defi}
	
	\begin{defi}
		\label{def bubble level}
		\textnormal{We define the {\it level of a blowup sequence} and its corresponding bubble
			as follows.\\
			(1) We say a blowup sequence $\{ (x_{k},r_{k}) \} $ of $\{ u_{k} \}$ at the point 0 is at 1-level  if there does not exist any other blowup sequence $\{ (x_k',r_k') \}$ with $\{ (x_k,r_k) \} < \{ (x_k',r_k') \}.$ \\
			(2) For $m \geq 2$, we define a blowup sequence  at $m$-level inductively: a blowup sequence $\{ (x_{k},r_{k}) \} $ is called  at $m$-level if there exists a blowup sequence $\{ (x_k',r_k') \} $ which is at $(m-1)$-level such that $\{ (x_{k},r_{k}) \} $ is right on the top of $\{ (x_k',r_k') \} $. 	\\
			(3) For $m \geq 1$, a bubble is called at $m$-level if it is induced by a blowup sequence  at $m$-level. }
	\end{defi}
	
	We point out that two essentially different blowup sequences at $m$-level may be  right on the top of two essentially different blowup sequences at $(m-1)$-level.

	By Remark \ref{remark different bubble relationship}, if $\{ (x_k^1,r_k^1) \},\{ (x_k^2,r_k^2) \}$ are at the same level, then  for any $t>0$,
	$$
	D_{tr_k^1}(x_k^1)\cap D_{tr_k^2}(x_k^2) =\emptyset
	$$
	when $k$ is sufficiently large.

	When $g_k$ is defined on $D \backslash D_{r_k}(x_k)$ with $r_{k} \rightarrow 0$, we introduce 
	
	\begin{defi}\textnormal{
			Let $g_k=e^{2u_k}g_{\euc}\in\mathcal{M}(D\backslash D_{r_k}(x_k))$ with $r_k\rightarrow 0$, $x_k\rightarrow x_0\in D$.
			We say {\it $u_{k}$ has a bubble $u$ near $x_{0}$ on $D \backslash D_{r_k}(x_k)$} if there exist $t_{k} \rightarrow 0$, $y_{k} \rightarrow x_{0}$, 
			$$
			\frac{r_k}{t_k}\rightarrow 0,\s \textnormal{or}\s \frac{|x_k-y_k|}{r_k}\rightarrow\infty
			$$
			such that $u_{k}(t_{k}x+y_{k})+\log t_{k}$ converges weakly to $u$ in $W_{\loc}^{1,p}(\mathbb{R}^{2})$. }
	\end{defi}

	\subsection{Residue of curvature measure }
	
	Let $ g=e^{2u}g_{\euc} \in \mathcal{M}( D \backslash \{0\}  )$ with 
	\begin{align*}
	|\mathbb{K}_{g}|( D \backslash \{0\}  ) <\pi.  
	\end{align*}
 We extend $\K_g$ to a signed Radon measure $\mu$ by taking $\mu(A)=	\mathbb{K}_{g}(A \cap (D \backslash \{0\})), \forall A \subset \mathbb{R}^{2}$ and we write
	$\mu=\K_g\lfloor (D\backslash\{0\})$.
	Let $I_{\mu}$ be defined by
	\begin{align}
	\label{def of I mu}
	I_{\mu}(x)=-\frac{1}{2 \pi} \int_{\mathbb{R}^{2}} \log |x-y| d \mu(y),
	\end{align}
	then  by a result of Brezis-Merle \cite{Brezis1991} (see also \cite[Proposition 2.2]{C-L2}),
	$I_\mu\in\cap_{q \in [1,2)}W^{1,q}_{\loc}(\R^2)$, $ -\Delta I_{\mu}=\mu$ and
	\begin{align}
	\label{Integral estimate of Imu}
	\int_{D_{R}}  e^{  \frac{ (4 \pi-\epsilon)|I_{\mu}|}{ |\mu|(\mathbb{R}^{2})  }    } \leq CR^{  \frac{\epsilon}{2 \pi}   }, \quad \forall R>0 \text{ and } \epsilon \in (0,4 \pi).
	\end{align}
	By (\ref{Gauss-Bonnet.at.singularity1}), 
	$$
	\lim_{{\rm a.e}\,\,r\rightarrow 0}\int_{\partial D_r}\frac{\partial I_{\mu}}{\partial r}=0.
	$$
	Put
	$$
	\lambda=\lim_{{\rm a.e}\,\,r\rightarrow 0}\frac{1}{2 \pi }\int_{\partial D_r}\frac{\partial u}{\partial r},
	$$
	and
	$$
	w(z)=u(z)-I_\mu(z)-\lambda\log |z|.
	$$
	Then $w$ is harmonic on $D\backslash \{0\}$ with $\int_{\partial D_t}\frac{\partial w}{\partial t}=0$. Hence, we can find a holomorphic function $F$ on $D\backslash\{0\}$ with
	$\Re(F)=w$, which means
	%Since $u-I_{\mu}$ is harmonic on $D \backslash \{0\}$, then %by [\cite{conway}Theorem 15.1.3],
	\begin{align*}
	(u-I_{\mu})(z)=\Re(F(z))+\lambda \log |z|, \quad z \in D \backslash \{0\}.
	\end{align*}
	%where $F$ is holomorphic on $D\backslash\{0\}$. We set $w=Re(F)$. 
	%then $ \int_{ \partial D_{r}} \frac{\partial w}{\partial r}=0$. 
	
	\begin{lem}
		\label{lemma removable singularity}
		\textnormal{(1)} If the area of $D$ is finite, namely,
		\begin{align*}
		\int_{D} e^{2u}<\infty,
		\end{align*}
		then $w$ is smooth on $D$, and $\lambda \geq -1$. 
		Moreover, $g\in \mathcal{M}(D)$ with
		$$
		\K_g=\mu-2\pi\lambda\delta_0.
		$$
		\textnormal{(2)} If we further assume $\mathbb{K}_{g} \geq 0$ on $D \backslash \{0\}$, then $\lambda>-1$. 
	\end{lem}
	
	\begin{proof}
		\textbf{Step 1}: We show $F$ is smooth. 
		
		Suppose that $0$ is a singularity of $F$. Then 0 is an essential singularity of $e^{sF}$ for any positive $s$. 
		Fix some $q \in (1,2)$ and choose $2m>\lambda q$. Then
		\begin{align*}
		\int_{D} r^{\lambda q   } e^{qw} \geq \int_{D} |z^{m} e^{  \frac{qF}{2}   }|^{2}=\infty,
		\end{align*}
		where we use the fact that 0 is a singularity of $ z^{m} e^{ \frac{qF}{2}   }$. However, 
		\begin{align*}
		\int_{D} r^{ \lambda q } e^{qw}=\int_{D} e^{qu}e^{-q I_{\mu}} \leq \left( \int_{D} e^{2u} \right)^{\frac{q}{2}} \left( \int_{D} e^{ \frac{2q}{2-q} |I_{\mu}|   } \right)^{\frac{2-q}{2}}
		\end{align*}
		leads to a contradiction by choosing $q$ such that $ \frac{2q}{2-q} < \frac{3 \pi }{  |\mu|(\mathbb{R}^{2})}$ and using (\ref{Integral estimate of Imu}). Thus, $F$ is smooth on $D$, and so is $w$. Now $u$ solves the following equation weakly
		\begin{align*}
		-\Delta u= \mu-2 \pi \lambda \delta_{0}.
		\end{align*}
		
		\textbf{Step 2}: We show that $\lambda \geq -1$. 
		
		\begin{comment}
		If $\lambda <-1$, 
		then we choose $\delta \in (0,1)$ such that $2 \delta \lambda<-1$.  Then 
		\begin{align*}
		\infty&=\int_{D} r^{ 2 \delta \lambda  }=\int_{D} r^{ 2 \delta \lambda  } e^{2 \delta v} e^{-2\delta v} \\
		& \leq \left(  \int_{D} \left(  r^{2 \delta \lambda} e^{2 \delta I_{\mu}}   \right)^{\frac{1}{\delta}}   \right)^{\delta} \left(  \int_{D} e^{ \frac{2 \delta| I_{\mu} | }{1-\delta}    } \right)^{1-\delta}\\
		& \leq \left(  \int_{D} e^{2u-2w} \right)^{\delta }\left(\int_{D}  e^{ \frac{ 2 \delta |I_{\mu}|}{1-\delta}   } \right)^{1-\delta}.
		\end{align*}
		
		We choose $\delta=\frac{1}{4}-\frac{1}{4 \lambda}$, then
		$
		\frac{2 \delta}{1-\delta}=\frac{2 \lambda-2}{3 \lambda+1}<\frac{4 \pi-\epsilon_{1}}{\epsilon_{1}},
		$
		leading to a contradiction by using (\ref{Integral estimate of Imu}) again. 
		\end{comment}
				If $\lambda <-1$, we choose $\delta=\frac{1}{2}-\frac{1}{2 \lambda} \in (-\frac{1}{\lambda},1)$. Since $|\mathbb{K}_{g}|(D \backslash  \{0\})<+\infty$, 
				we can  select $\epsilon_{1} \in (0, \pi)$ and $\rho\in (0,1)$ such that
				\begin{align*}
				\frac{2 \delta}{1-\delta}<\frac{4 \pi-\epsilon_{1}}{\epsilon_{1}}, \quad |\mathbb{K}_{g}|(D_{\rho} \backslash \{0\})<\epsilon_{1}. 
				\end{align*}
			We set a measure $\mathbb{K}_{g}^{\rho}$ defined on $D_{\rho} \backslash  \{0\}$ by $\mathbb{K}_{g}^{\rho}=\mathbb{K}_{g}$, and a measure $\mu^{\rho}$ defined on $\mathbb{R}^{2}$ by 
			\begin{align*}
			\mu^{\rho}(A)=\mathbb{K}_{g}^{\rho}( A \cap (D_{\rho} \backslash \{0\}) ), \quad \forall A \subset \mathbb{R}^{2}, 
			\end{align*}
			which means that $|\mu^{\rho}|(\mathbb{R}^{2}) \leq \epsilon_{1}$. 
			Similar to (\ref{def of I mu}), we define
			\begin{align*}
			I_{\mu}^{\rho}(x)=-\frac{1}{2 \pi} \int_{\mathbb{R}^{2}} \log |x-y| d \mu^{\rho}(y),
			\end{align*}
			then by (\ref{Integral estimate of Imu}), we have
			\begin{align}
			\label{I mu rho integral}
			\int_{D_{\rho}}  e^{  \frac{ (4 \pi-\epsilon_{1})|I_{\mu}^{\rho}|}{ \epsilon_{1} }    }  \leq \int_{D_{\rho}}  e^{  \frac{ (4 \pi-\epsilon_{1})|I_{\mu}^{\rho}|}{ |\mu^{\rho}|(\mathbb{R}^{2})  }    } \leq C \rho^{  \frac{\epsilon_{1}}{2 \pi}   } \leq C. 
			\end{align}
			We set
			\begin{align*}
			w^{\rho}(z)=u(z)-I_{\mu}^{\rho}(z)-\lambda \log |z|,
			\end{align*}
			then by \textbf{Step 1}, $w^{\rho}$ is smooth on $D_{\rho}$. By the choice of $\delta$, $\delta \lambda <-1$, then we obtain
				\begin{align*}
			\infty&=\int_{D_{\rho}} r^{ 2 \delta \lambda  }=\int_{D_{\rho}} r^{ 2 \delta \lambda  } e^{2 \delta I_{\mu}^{\rho}} e^{-2\delta I_{\mu}^{\rho}} \\
			& \leq \left(  \int_{D_{\rho}} \left(  r^{2 \delta \lambda} e^{2 \delta I_{\mu}^{\rho}}   \right)^{\frac{1}{\delta}}   \right)^{\delta} \left(  \int_{D_{\rho}} e^{ \frac{2 \delta| I_{\mu}^{\rho} | }{1-\delta}    } \right)^{1-\delta}\\
			& \leq \left(  \int_{D_{\rho}} e^{2u-2w^{\rho}} \right)^{\delta }\left(\int_{D_{\rho}}  e^{ \frac{ 2 \delta |I_{\mu}^{\rho}|}{1-\delta}   } \right)^{1-\delta},
			\end{align*}
			which leads to a contradiction by the choice of $\epsilon_{1}$ and (\ref{I mu rho integral}). 

		\textbf{Step 3}: Now we further assume that  $\K_{g}\geq 0$ on $D\backslash\{0\}$. 
		
		If $\lambda=-1$, we set
		$$
		u_\tau(x)=u(\tau x)+\log \tau,
		$$
		then using the fact that: $u(x)=w(x)+I_{\mu}(x)- \log |x|$, we have
		\begin{align*}
		u_{\tau}(x)=I_\mu(\tau x)+w(\tau x)-\log|x|.
		\end{align*}
		It is easy to check 
		$$
		\lim_{\tau\rightarrow 0^+}\int_{D_2\backslash D_1}e^{2u_\tau}=\lim_{\tau\rightarrow 0^+}\int_{D_{2\tau}\backslash D_\tau}e^{2u}=0.
		$$
		By (\ref{def of I mu}) and $\K_{g}\geq 0$, for any $x \in D$, 
		$$
		I_{\mu}(x) \geq  -\frac{1}{2\pi}\int_{\{|x-y|\geq 1\}}\log |x-y| d\mu(y) \geq - \log 2.
		$$
		We fix arbitrary $\tau<\frac{1}{2}$. 
		Then for any $t\in (1,2)$, $\tau t <1$ and 
		$$
		u_\tau|_{\partial D_t}
		= (I_\mu+ w)|_{\partial D_{\tau t}}-\log t>\min_{D} w-2 \log 2,
		$$
		which yields that
		\begin{align*}
		\int_{D_{2} \backslash D_{1} } e^{2u_{\tau}} \geq (3 \pi) e^{ 2 (  \min_{D} w-2\log 2     )     }.
		\end{align*}
		We get a contradiction.
	\end{proof}
	
	We observe that for $g=e^{2u}g_{\euc}\in\mathcal{M}(D \backslash \{0\})$ satisfying
	\begin{align*}
	|\mathbb{K}_{g}|(D \backslash \{0\})+\int_{D} e^{2u}<\infty,
	\end{align*}
	then by the statement (1) of Lemma \ref{lemma removable singularity}, we can extend $g \in \mathcal{M}(D)$. 
	
	For  $g=e^{2u}g_{\euc}\in\mathcal{M}(\R^2\backslash D)$ satisfying
	\begin{align*}
	|\K_g|(\R^2\backslash D)+\int_{ \R^2\backslash D } e^{2u}<\infty,
	\end{align*}
	we set 
	\begin{align}
	\label{transformation}
	x'=\frac{x}{|x|^2}, \quad u'(x')=u(\frac{x'}{|x'|^2})-2\log|x'|,
	\end{align}
	then $g'=e^{2u'}g_{\euc}\in\mathcal{M}(D\backslash\{0\})$, and 
	\begin{align*}
	|\K_{g'}|(D\backslash\{0\})+\int_{D \backslash \{0\}} e^{2u'}=|\K_{g}|(\R^2\backslash D)+\int_{ \R^2\backslash D } e^{2u}<\infty. 
	\end{align*}
	Then by (1) in Lemma \ref{lemma removable singularity} again, we can extend $g' \in \mathcal{M}(D)$. We introduce	
	
	\begin{defi}
		\label{def residue}\textnormal{
			(1)
			For $g=e^{2u}g_{\euc}\in\mathcal{M}(D)$, the {\it residue} of $g$ (or $u$) at $0$ is 
			$$
			\res(g,0)=\res(u,0)=-\frac{1}{2\pi}\K_g(\{0\}).
			$$
		}
		\textnormal{(2) For $g=e^{2u}g_{\euc}\in\mathcal{M}(D \backslash \{0\})$ satisfying
			\begin{align*}
			|\mathbb{K}_{g}|(D \backslash \{0\})+\int_{D} e^{2u}<\infty,
			\end{align*}
			%we can extend $g$ such that $g \in\mathcal{M}(D)$ by Lemma \ref{lemma removable singularity}, and we still denote the extension by $g$. Then
			the residue of $g$ (or $u$) at $0$ is 
			$$
			\res(g,0)=\res(u,0)=-\frac{1}{2\pi}\K_g(\{0\}).
			$$
		}
		\textnormal{ (3) 
			For  $g=e^{2u}g_{\euc}\in\mathcal{M}(\R^2\backslash D)$ satisfying
			\begin{align*}
			|\K_g|(\R^2\backslash D)+\int_{ \R^2\backslash D } e^{2u}<\infty,
			\end{align*}
			the residue of $g$ (or $u$) at $\infty$ is 			
			\begin{align*}
			\res(g,\infty)=\res(u,\infty)=-\frac{1}{2\pi}\K_{g'}(\{0\})
			\end{align*}
			where for $g'=e^{2u'}g_{\euc}\in\mathcal{M}(D\backslash\{0\})$ with $u'$ is as in (\ref{transformation}) is extended in $\mathcal{M}(D)$. 
		}
	\end{defi}
	Since
	$$
	\lim_{a.e.\, r\rightarrow 0}\int_{\partial D_r}\frac{\partial  u'}{\partial r}=2\pi\res(g',0)=2\pi\res(u,\infty),
	$$
	then by direct calculations,
	$$
	\lim_{a.e.\, r\rightarrow 0}\int_{\partial D_\frac{1}{r}}\frac{\partial u}{\partial r}=-2\pi(2+\res(u,\infty)).
	$$

	\begin{pro}
		\label{prop non-positive on neck} 
		Let $g_k=e^{2u_k}g_{\euc}\in\mathcal{M}(D)$ with
		\begin{align*}
		|\K_{g_k}|(D)+\Area(D,g_k)\leq\Lambda.
		\end{align*}
		Assume that $u_k$ converges to $u$ weakly in $W^{1,p}(D)$ for some $p\in[1,2)$
		and $\{u_k\}$ has a blowup sequence $\{(x_k,r_k)\}$ at $0$ with the corresponding bubble $u'$. Then there exists $t_i\rightarrow 0$, such that
		\begin{align*}
		\lim_{i\rightarrow\infty}\lim_{k\rightarrow\infty}\K_{g_k}(D_{t_i}(0)\backslash D_{ \frac{r_k}{t_i}}  (x_k))=-2 \pi \left(2+\Res(u,0)+\Res(u',\infty)\right).
		\end{align*}
	\end{pro}
	
	\begin{proof}
		We may assume
		$$
		\lim_{r\rightarrow 0,r\in\mathcal{A}}\int_{\partial D_r}\frac{\partial u}{\partial r}=2\pi\res(u,0),
		$$
		and
		$$
		\lim_{r\rightarrow 0,r\in\mathcal{A}}\int_{\partial D_\frac{1}{r}}\frac{\partial u'}{\partial r}=-2\pi(2+\res(u',\infty)),
		$$
		where  $\mathcal{A}$ is the set chosen as in subsection 2.2.
		By \eqref{convergence.u_r}, we can find $t_i\rightarrow 0$, such that $t_i \in\mathcal{A}$ 
		and 
		$$
		\lim_{k\rightarrow\infty}\int_{\partial D_{t_i}(0)}\frac{\partial u_k}{\partial r}=\int_{\partial D_{t_i}(0)}\frac{\partial u}{\partial r} \s \textnormal{and}\s\lim_{k\rightarrow\infty}\int_{\partial D_{\frac{r_k}{t_i}} (x_k)}\frac{\partial u_k}{\partial r}=\int_{\partial D_{\frac{1}{t_i}}(0)}\frac{\partial u'}{\partial r}.
		$$
		Then 
		\begin{align*}
		\lim_{i\rightarrow\infty}\lim_{k\rightarrow\infty}\K_{g_k}(D_{t_i}(0)\backslash D_{\frac{r_k}{t_i}}(x_k))=&\lim_{i\rightarrow\infty}\Big(\int_{\partial D_{\frac{1}{t_i}}(0)}\frac{\partial u'}{\partial r}-\int_{\partial D_{t_i}(0)}\frac{\partial u}{\partial r}\Big)\\
		=&-2 \pi \big(2+\res(u,0)+\res(u',\infty)\big).
		\end{align*}
	\end{proof}

	\section{Area identity}
	
	The objective of this section is to analyze the area in various regions arose in a blowup procedure. The main observation is the conservation of area in the limiting process when curvatures are uniformly bounded away from zero. 
	
	Throughout this section, we assume $g_k=e^{2u_{k}} g_{\euc} \in \mathcal{M}(D)$ satisfying: 
	
	\begin{itemize}
		\item[(P1)] \( \K_{g_k} = f_k dV_{g_k} + \lambda_k \delta_0 \) (where \( \lambda_k \) might be 0) with \( f_k \geq 1 \) or \( f_k \leq -1 \);
		\item[(P2)] \( |\K_{g_k}|(D) + \Area(D,g_k) \leq  \Lambda_{1} \);
		\item[(P3)] \( r^{-1} \|\nabla u_{k}\|_{L^{1}(D_r(x))} \leq \Lambda_{2} \), for all \( D_r(x) \subset D \).
	\end{itemize}
	
	Let $c_{k}$ be the mean value of $u_{k}$ on $D_{\frac{1}{2}}$, then by the Poincaré and Sobolev inequalities and (P3),  after passing to a subsequence, $u_{k}-c_{k}$ converges weakly in $\cap_{p \in [1,2)} W^{1,p}(D_{\frac{1}{2}})$. By (P2) and Jensen's inequality, $c_{k}<+\infty$. 
	If $c_{k}$ is bounded below, we may assume  $u_{k}$ converges weakly in $\cap_{p \in [1,2)} W^{1,p}(D_{\frac{1}{2}})$ to  some $u$ and define $g_{\infty}=e^{2u} g_{\euc}$. If $c_{k} \rightarrow -\infty$, we define $\Area(D_{\frac{1}{2}},g_{\infty})=0$. 
	
	The main result in this section is the following area identity.
	
	\begin{theorem}
		\label{theorem area convergence of bubble tree}
		Let $g_{k} \in \mathcal{M}(D)$ satisfy conditions \textnormal{(P1)-(P3)}.
		We assume $\{u_{k}\}$ has finitely many bubbles $v_{1}, \cdots, v_{m}$, induced by blowup sequences $\{(x_{k}^{1},r_{k}^{1}) \}$, $\cdots, \{ (x_{k}^{m},r_{k}^{m})\}$ at the point $0$, respectively. and we further assume that $\{u_{k}\}$ has no bubbles at any point $x \in D \backslash \{0\}$. Then after passing to a subsequence, 
		\begin{align}
		\label{bubble tree area convergence}
		\lim_{k \rightarrow \infty} \Area( D_{{1}/{2}},g_{k})=\Area(D_{{1}/{2}},g_{\infty})+\sum_{i=1}^m  \Area(\mathbb{R}^{2},e^{2v_{i}}g_{\euc}).
		\end{align}
		We allow $m=0$, which means $\{u_{k}\}$ has no bubbles at $0$, and the sum term in 	(\ref{bubble tree area convergence}) vanishes for this case. 
	\end{theorem}
	
	Before proving Theorem \ref{theorem area convergence of bubble tree}, we will derive several results that will be employed in our proof. We start from a {\it 3-circle lemma} on a cylinder:  
	
	\begin{lem}\label{lemma 3 circle lemma}
		Let $\Lambda, \kappa$ and $L>1$ be positive constants. Suppose 
		\begin{align}
		\label{3cir1}
		\|\nabla u \|_{ L^{1}(S^{1} \times [t,t+1])   } < \Lambda
		\end{align}
		for $a.e.$  $t \in [-L,4L-1]$, 
		and 
		\begin{align}
		\label{3cir2}
		\int_{ S^{1} \times \{t\}} \frac{ \partial u}{ \partial t }<-2 \pi \kappa<0 \quad \big( \text{or }  	\int_{ S^{1} \times \{t\}} \frac{ \partial u}{ \partial t }>2 \pi \kappa>0     \big).
		\end{align}
		We further assume $ L > \frac{16 \Lambda }{\kappa}$. Then there exists $\tau=\tau(\kappa,\Lambda)>0$ such that if 
		\begin{align*}
		| \mathbb{K}_{g}|(S^{1} \times [-L,4L]) < \tau, 
		\end{align*}
		then
		\begin{align}
		\label{3cir3}
		\Area( Q_{2},g )< e^{-\frac{ \kappa L}{2}  }	\Area(  Q_{1},g ) \quad \big(\text{or } \Area(Q_{1} ,g )< e^{-\frac{ \kappa L}{2}  } 	\Area(Q_{2} ,g) \big), 
		\end{align}
		where $Q_{i}=S^{1} \times [(i-1)L,iL]   ), i=1,2.$
	\end{lem}
	
	\begin{proof}
		Assume (\ref{3cir3}) is not true. Then we can find $ g_{k}=e^{2u_{k}}(dt^{2}+d \theta^{2}    )$, such that
		\begin{align*}
		| \mathbb{K}_{g_{k}}|(Q) \rightarrow 0 \text{ as } k \rightarrow \infty, \quad \Area( Q_{2},g_{k} )\geq e^{-\frac{ \kappa L}{2}  }	\Area(  Q_{1},g_{k} ) \quad,
		\end{align*}
		where we set $ Q=S^{1} \times [-L,4L]$. 
		
		Let $c_{k}$ be the mean value of $u_{k}$ over $Q$. Then, by the Poincaré inequality
		and the Sobolev embedding theorem,  $u_{k}-c_{k}$ converges weakly to a harmonic function $v$ in $\cap_{q \in [1,2)} W^{1,q}(Q)$. Since $\nabla v$ is also harmonic function in $Q$, then by the mean value theorem and (\ref{3cir1}), there exists $t$ such that
		\begin{align}
		\max_{ S^{1} \times [0,3L]    } |\nabla v| \leq \frac{1}{|D_{\frac{1}{2}}|  } \|\nabla v\|_{  L^{1}( S^{1} \times [t,t+1]  )   } 
		\leq  \frac{4}{\pi} \lim_{k \rightarrow \infty} \| \nabla (u_{k}-c_{k})\|_{  L^{1}( S^{1} \times [t,t+1]  )   }  \leq \frac{4\Lambda}{\pi}	\label{3cir4}. 
		\end{align}
		
		The harmonic function $v$ can be expanded as 
		\begin{align*}
		v(t,\theta)=A+Bt+\sum_{k=1}^{\infty} ( a_{k}(t) \cos(k \theta)+ b_{k}(t) \sin (k \theta)     ):=A+Bt+\tilde{v}(t,\theta). 
		\end{align*}
		By direct calculations, 
		\begin{align*}
		A+Bt=\frac{1}{2\pi}\int_{ S^{1} \times \{t\}} v , \quad B=\frac{1}{2 \pi} \int_{ S^{1} \times \{t\}} \frac{\partial v}{\partial t}. 
		\end{align*}
		By (\ref{3cir4}), for any $t \in [0,3L]$,  $ |B| \leq \frac{4 \Lambda }{\pi}$ and 
		\begin{align*}
		|\tilde{v}(t,\theta)|&=\left|v(t,\theta)-A-Bt\right|=\left|v(t,\theta)-\frac{1}{2 \pi } \int_{ S^{1} \times \{t\}} v\right|\\
		&=\left|\frac{1}{2 \pi} \int_{\theta-\pi}^{\theta+\pi} (  v(t,\theta)-v(t,\phi)   ) d \phi\right|\\
		& \leq \pi \max_{ \theta} \left|  \frac{\partial v}{\partial \theta}\right| \leq 4 \Lambda. 
		\end{align*}
		In particular, by (\ref{3cir2}), 
		\begin{align*}
		B=\frac{1}{2 \pi}  \int_{S^{1} \times [0,1]} \frac{\partial v}{\partial t}=\lim_{k \rightarrow \infty} \frac{1}{2 \pi} \int_{S^{1} \times [0,1]} \frac{\partial (u_{k}-c_{k})}{\partial t}<-\kappa. 
		\end{align*}
		
		Set $ g_{\infty}=e^{2v}( dt^{2}+d \theta^{2}  )$ and $ \tilde{g}_{\infty}=e^{2A+2Bt}( dt^{2}+d \theta^{2}  )$, so $ g_{\infty}=e^{2 \tilde{v}} \tilde{g}_{\infty}$. Then 
		\begin{align*}
		e^{-8 \Lambda} \tilde{g}_{\infty} \leq  g_{\infty} \leq e^{8 \Lambda} \tilde{g}_{\infty} . 
		\end{align*}
		It follows 
		\begin{align*}
		\Area( Q_{i}, \tilde{g}_{\infty}    )& =\int_{0}^{2 \pi } \int_{(i-1)L}^{iL} e^{2A+2Bt} dt d \theta \\
		& =2 \pi e^{2A} \cdot \frac{1}{2B} e^{2Bt}|_{(i-1)L }^{iL} \\
		&=\frac{\pi}{B} e^{ 2A+2B(i-1)L     }(   e^{2BL}-1    ), 
		\end{align*}
		which implies 
		\begin{align*}
		\frac{\Area( Q_{2}, g_{\infty} )}{ \Area(Q_{1}, g_{\infty})  } \leq e^{16 \Lambda } 	\frac{\Area( Q_{2}, \tilde{g}_{\infty} )}{ \Area(Q_{1}, \tilde{g}_{\infty}  )} \leq e^{16 \Lambda + 2BL} \leq e^{16 \Lambda - 2 \kappa L} \leq e^{-\kappa L}. 
		\end{align*}
		On the other hand, by Theorem \ref{thm convergence of solutions of Radon measure when measure is small},
		\begin{align*}
		e^{-\frac{\kappa L}{2}} \leq \frac{ \Area(Q_{2},g_{k})}{ \Area(Q_{1},g_{k})  }=\frac{ \Area(Q_{2},e^{-2c_{k}}g_{k})}{\Area(Q_{1},e^{-2c_{k}}g_{k})  } \rightarrow \frac{\Area( Q_{2}, g_{\infty} )}{ \Area(Q_{1}, g_{\infty})  }, 
		\end{align*}
		which leads to a contradiction. 
	\end{proof}
	
	Assume that there exist
	$\Lambda_{1}, \Lambda_{2}>0$ such that
	$g=e^{2u}g_{\euc}  \in \mathcal{M}(D)$ satisfies  
	\begin{align*}
	|\mathbb{K}_{g}|(D) \leq \Lambda_{1} \s \textnormal{and} \s
	r^{-1} \|\nabla u\|_{ L^{1}(D_{r}(x))   } \leq \Lambda_{2}, \text{ for all } D_{r}(x) \subset D.
	\end{align*}
	Let $t=-\log r \in (0,\infty)$. Then 
	\begin{align*}
	e^{2u(r,\theta)} (dr^{2}+r^{2} d \theta^{2})=e^{  2u( e^{-t},\theta)} r^{2} \left( dt^{2}+d \theta^{2}\right)=e^{ 2u( e^{-t},\theta) -2t}( dt^{2}+d \theta^{2}  ). 
	\end{align*}
	Set 
	\begin{align}
	\label{uk coordinte to vk}
	v(t,\theta)=u(  e^{-t},\theta)-t. 
	\end{align} 
	Then
	\begin{align*}
	\int_{ S^{1} \times \{ t_{0}\} }& \frac{ \partial v}{\partial t}= \int_{0}^{2 \pi } \frac{ \partial (u(  e^{-t},\theta   )-t)}{\partial t} |_{t=t_{0}} d \theta=\int_{0}^{2 \pi } \left(-\frac{\partial u( e^{-t_{0}}, \theta  )}{\partial r} e^{-t_{0}} -1 \right) d \theta \\
	&= -\int_{ \partial D_{ e^{-t_{0}}  }} \frac{\partial u}{\partial r} -2 \pi
	\end{align*}
	and 
	\begin{align}
	\|\nabla v \|_{ L^{1} ( S^{1} \times [t_{0},t_{0}+1]   )   } & \leq \int_{ 0}^{2 \pi } \int_{t_{0}}^{t_{0}+1 } \left( \left|\frac{\partial u}{\partial r}\right| e^{-t}+\left|\frac{\partial u}{\partial \theta}\right| \right) dt d \theta+2 \pi \nonumber\\
	&= \int_{0}^{2 \pi } \int_{ e^{-t_{0}-1}}^{e^{-t_{0}}} \left( \left|\frac{\partial u}{\partial r}\right|+\frac{1}{r} \left|\frac{\partial u}{\partial \theta}\right| \right) drd \theta +2 \pi \nonumber\\
	&\leq  \frac{1}{ e^{-t_{0}-1}  } \int_{D_{e^{-t_{0}}} \backslash D_{e^{-t_{0}-1}   }}  |\nabla u| +2 \pi \nonumber\\
	&\leq \frac{1}{ e^{-t_{0}-1}  } \int_{D_{e^{-t_{0}}}}  |\nabla u| +2 \pi \nonumber \\
	&\leq C \Lambda_{2} +2 \pi:=\Lambda_{3}. \label{finite gradient under different coordinates}
	\end{align}
	Clearly 		
	$$
	|\mathbb{K}_{g}|(  S^{1} \times [t_{0},t_{0}+1]   ) =|\mathbb{K}_{g}|(   D_{e^{-t_{0}}   } \backslash D_{e^{-t_{0}-1}   } ).
	$$

	Using Lemma \ref{lemma 3 circle lemma}, the above analysis leads to estimates on ratio of areas on consecutive annuli, phrased in cylindrical coordinates: 
	\begin{cor}\label{remark 3 circle on annulus}	
		Let $g=e^{2u} g_{\euc} \in \mathcal{M}(D)$ which satisfies \textnormal{(P2)-(P3)}. 
		Let $\kappa>0$ be a constant and $L >\frac{16  \Lambda_{3}}{\kappa}$, where $\Lambda_{3}$ is chosen as in (\ref{finite gradient under different coordinates}). 
		If for a.e. $t \in [ L,6L-1]$, 
		\begin{align*}
		-\int_{\partial D_{e^{-t}  }} \frac{\partial u}{\partial r}-2 \pi <-2\pi \kappa<0 \quad (\text{or }  -\int_{\partial D_{e^{-t}  }} \frac{\partial u}{\partial r}-2 \pi >2\pi \kappa>0     ). 
		\end{align*}
		Then there exists $\tau=\tau( \Lambda_{1},\Lambda_{2},\kappa )>0$ such that if
		\begin{align*}
		| \mathbb{K}_{g}|(  D_{ e^{-L}    }  \backslash  D_{ e^{-6L}    }     )< \tau, 
		\end{align*}
		then
		\begin{align*}
		&\Area( D_{ e^{-3L} } \backslash D_{ e^{-4L} } ,g ) < e^{  -\frac{\kappa L }{2}   } \Area( D_{ e^{-2L} } \backslash D_{ e^{-3L} } ,g )\\
		(\text{or }  &\Area( D_{ e^{-3L} } \backslash D_{ e^{-4L} } ,g  ) > e^{  \frac{\kappa L }{2}   } \Area( D_{ e^{-2L} } \backslash D_{ e^{-3L} } ,g  )   ). 
		\end{align*}
	\end{cor}
	
	Applying the above 3-circle lemma, we can deduce the following criterion for vanishing of area in  the neck region during the blowup.
	\begin{lem}
		\label{lem area identity on on neck region}
		Let $g_{k}=e^{2u_{k}} g_{\euc} \in \mathcal{M}(D)$ which satisfies \textnormal{(P1)-(P3)} with $\lambda_k=0$. 
		We assume $x_{k} \rightarrow 0, r_{k} \rightarrow 0$ and 
		\begin{align}
		\label{area identity neck sup}
		\lim_{\rho \rightarrow 0} \lim_{k \rightarrow \infty} \sup_{ r \in [\frac{r_{k}}{\rho}, \rho]}\Area( D_{r}(x_{k}) \backslash D_{ \frac{r}{e} }(x_{k}), g_{k}    )=0.
		\end{align}
		Then
		\begin{align}
		\label{area identity neck whole}
		\lim_{\rho \rightarrow 0} \lim_{k \rightarrow \infty} \Area(D_{\rho}(x_{k}) \backslash D_{\frac{r_{k}}{\rho} }(x_{k}), g_{k})=0. 
		\end{align}
	\end{lem}
	
	\begin{proof} 
		By change of coordinates: $x_{k}+re^{i \theta} \rightarrow (\theta, T)=(\theta,-\log r)$, for any $r \leq {1}/{e}$ and large $k$, 
		we can view $g_{k}|_{ ( D_{ r}(x_{k}) \backslash D_{ \frac{r_{k}}{r}  }(x_{k}))}$ as a metric on $S^{1} \times [T,T_{k}-T ]$, where $T_{k}=-\log r_{k}$ and $g_{k}=e^{2v_{k}  }( d \theta^{2}+dt^{2})$ with $v_{k}$ as in (\ref{uk coordinte to vk}). 
		By (\ref{area identity neck sup}), we have
		\begin{align}
		\label{area identity tube sup}
		\lim\limits_{T \rightarrow \infty} \lim\limits_{k \rightarrow \infty} \sup\limits_{ t \in [T,T_{k}-T]} \Area( S^{1} \times [t,t+1], g_{k}  )=0.
		\end{align}
		
		Our goal is to show 
		\begin{align}
		\label{area identity tube whole}
		\lim_{T \rightarrow \infty} \lim_{k \rightarrow \infty}\Area( S^{1}\times [T,T_{k}-T] , g_{k})=0. 
		\end{align} 
		It suffices to show that for any fixed $\delta > 0$ there is $T_0$ such that for any large $k$, 
		\begin{align*}
		{\Area}(S^{1} \times [T_0, T_{k} - T_0], g_{k}) \leq 5\delta.
		\end{align*}
		
		Let $\Lambda_{3}$ and $\tau = \tau(\frac{\delta}{2\pi}, \Lambda_1, \Lambda_{2})$ be as in Corollary \ref{remark 3 circle on annulus} and select an $L > \frac{32\pi\Lambda_{3}}{\delta}$. 
		By (P1) and (P2), we can extend $\int_{ S^{1} \times \{t\} } \frac{\partial v_{k}}{\partial t}$ to a continuous function $\Phi_{v_{k}}(t)$ as in (\ref{continuous extension of phi}). 
		Now we divide the proof into two steps.
		
		\vspace{.1cm}
		\noindent{\bf Step 1:} $|\Phi_{v_{k}}(t)|>\delta$ for any $t$. 
		
		Then by the continuity of $\Phi_{v_k}$, either $\Phi_{v_k}<-\delta$ or $\Phi_{v_k}>\delta$. 
		Without loss of generality, we assume 
		$\Phi_{v_{k}}(t)<-\delta$. 
		By (P2), 
		$$
		\#\Big\{ i:|\K_{g_k}|(S^1\times[iL,(i+1)L])\geq\frac{\tau}{5}  \Big\}\leq \Big[ \frac{5 \Lambda_1}{\tau}\Big]+1:=\Lambda_{4}.
		$$
		
		Let $N$ be an integer such that
		\begin{align*}
		N \geq \max\Big\{  \frac{1}{1- e^{-16 \pi \Lambda_3  }  } , 2 \Big\}.
		\end{align*}
		
		By (\ref{area identity tube sup}), we may find $T_0$, such that for any sufficiently large $k$, 
		\begin{align}
		\label{choice of T0}
		\Area(S^1\times[t,t+L],g_{k})<\frac{\delta}{  16N \Lambda_{4} },\s \forall t\in[T_0,T_k-T_0].
		\end{align}

		We select integers $a_k,b_k$ such that 
		$$
		(a_k-1) L \leq T_{0} \leq a_kL \s \text{and}\s b_kL  \leq T_k-T_0 \leq (b_k+1) L
		$$ 
		and set
		$$
		\Big\{ i \in [a_{k},b_{k}-1]: |\K_{g_k}|(S^1\times[iL,(i+1)L])\geq\frac{\tau}{5}  \Big\}=\left\{  i_{1}^{k}, \cdots, i_{ m_{k}}^{k}    \right \},
		$$
		where $i_{1}^{k}<\cdots<i_{m_{k}}^{k}, m_{k} \leq \Lambda_{4}$. 
		Let $a_{k}=i_{0}^{k}+1, b_{k}=i_{m_{k}+1}^{k}$ and write
		$$
		[a_{k}L,b_{k}L]=\bigcup^{m_k}_{j=0}[(i^k_{j}+1)L, i^k_{j+1}L]\cup [ i^k_{j+1}L, (i^k_{j+1}+1)L]. 
		$$
		%$[a_{k}L, i_{1}^{k}L], [  i_{1}^{k}L, ( i_{1}^{k}+1   )L  ],[  ( i_{1}^{k}+1   )L ,  i_{2}^{k}L ],  [   i_{2}^{k}   L , (i_{2}^{k}+1)L ],  \cdots, [ i_{m_{k}}^{k}L, (i_{m_{k}}^{k}+1)L   ],  [(i_{m_{k}}^{k}+1)L , b_{k}L]. $
		
		For $j \in \{ 0, 1,\cdots, m_{k} \}$, we estimate $\Area(S^{1} \times [ (i_{j}^{k} +1)L, i_{j+1}^{k}L     ] , g_{k}     ) $ as follows:\\
		\noindent \textbf{Case 1}: 
		if $i_{j+1}^{k}-i_{j}^{k} < 6$, then
		\begin{align*}
		\Area(S^{1} \times [ (i_{j}^{k} +1)L, i_{j+1}^{k}L     ] , g_{k}     ) \leq \frac{ 4\delta}{  16 N \Lambda_{4}  } \leq \frac{ \delta}{  8\Lambda_{4}}; 
		\end{align*}
		\noindent \textbf{Case 2}: if $i_{j+1}^{k}-i_{j}^{k} \geq  6$, then by Lemma \ref{lemma 3 circle lemma},
		\begin{align*}
		\Area&(  S^{1} \times [ (i_{j}^{k} +1)L, i_{j+1}^{k}L], g_{k}    ) \\
		&\leq  \frac{2 \delta}{ 16 N\Lambda_{4}}+\Area(S^{1} \times [ (i_{j}^k+2)L,(i_{j+1}^k-1)L ],g_{k}) \\
		&\leq  \frac{ \delta}{  8 N\Lambda_{4}}+\frac{1}{ 1-e^{ -\frac{\delta L}{2}}}\Area(S^{1} \times [(i_{j}^k+2)L,(i_{j}^k+3)L ],g_{k}) \\
		&\leq  \frac{ \delta}{  8 N\Lambda_{4}}+\frac{ \delta}{ 16 \Lambda_{4}} \leq \frac{ \delta}{  8\Lambda_{4}}. 
		\end{align*}
		
		Therefore, we obtain that
		\begin{align*}
		\Area( S^{1} \times [T_{0},T_{k}-T_{0}]  ,g_{k} ) 
		&\leq \Area( S^{1} \times [a_k L,b_k L]  ,g_{k} ) +\frac{2 \delta}{  16N\Lambda_{4}} \\
		%\leq  & \sum_{j=1}^{m_{k}}\Area(  S^{1} \times [   i_{j}^{k}L, ( i_{j}^{k}+1)L ], g_{k})+\sum_{j=0}^{m_{k}}\Area(  S^{1} \times [  ( i_{j}^{k}+1)L,  i_{j+1}^{k}L     ]  , g_{k})+\frac{ \delta}{8N\Lambda_{4}} \\
		&\leq  \frac{3 m_{k} \delta} {  8 \Lambda_{4}  }+\frac{\delta}{  8N\Lambda_{4}} <\delta. 
		\end{align*}
		
		\noindent{\bf Step 2:}
		If there exists some $t$ such that $|\Phi_{v_{k}}(t)|\leq \delta$, 
		we still select $T_{0}$ as in \textbf{Step 1} and
		set
		\begin{align*}
		\alpha_{k}=\inf\left\{ t\in[T_{0},T_k-T_{0}]: |\Phi_{v_{k}}(t)| \leq \delta \right\}, \, \beta_{k}=\sup\left\{ t\in[T_{0},T_k-T_{0}]: |\Phi_{v_{k}}(t)| \leq \delta \right\}. 
		\end{align*}
		By (P1), we know that $f_{k} \geq 1$ or $f_{k} \leq -1$, then 
		\begin{align*}
		\Area(S^{1} \times [\alpha_{k}, \beta_{k}], g_{k}) &\leq \Big|\int_{ S^{1} \times [\alpha_{k}, \beta_{k}]} f_{k } dV_{g_{k}}\Big|=\big|\K_{g_{k}}( S^{1} \times [\alpha_{k}, \beta_{k}], g_{k})\big|\\ 
		&=|\Phi_{v_{k}}(\alpha_k)-\Phi_{v_{k}}(\beta_k)|\leq 2 \delta. 
		\end{align*}
		Since $|\Phi_{v_{k}}(t)| > \delta $ on $[ T_{0}, \alpha_{k}-L ] \cup [\beta_{k}+L,T_{k}-T_{0}]$, then by \textbf{Step 1}, 
		for large $k$, 
		\begin{align*}
		\Area(S^1\times[T_0,\alpha_k-L],g_k)<\delta,\s \Area(S^1\times[\beta_k+L,T_k-T_0],g_{k})<\delta.
		\end{align*}
		By (\ref{choice of T0}), for large $k$, 
		$$
		\Area(S^1\times[\max\{ \alpha_k-L,T_{0} \},\alpha_k],g_k)+\Area(S^1\times[\beta_k, \beta_k+L],g_k)<\delta. 
		$$
		Then we obtain
		\begin{align*}
		\Area( S^{1} \times [T_{0},T_{k}-T_{0}] , g_{k}) \leq 5 \delta.
		\end{align*}
	\end{proof}
	
	\begin{rem}\label{K cannot change sign} \textnormal{
			Let $g_k$ be a metric on $D_{t_k}\backslash D_{r_k}$ such that $r_k\rightarrow 0$ and $t_k/r_k\rightarrow +\infty$ and
			$(D_{t_k}\backslash D_{r_k},g_k)$ is isometric to $\frac{1}{k}S^1\times[0,k]$. Then $K_{g_k}=0$
			and $\Area(D_{t_k}\backslash D_{r_k},g_k)=2\pi$, which does not converge to $0$. Hence, the assumption that $K_{g_k}\geq 1$ or $K_{g_k}<-1$  cannot be removed.	}	
	\end{rem}
	
	The following proposition 
	shows that if no bubble occurs then there is no area concentration point. 
	
	\begin{pro}
		\label{pro no area concentration at zero on D}
		Let $g_{k}=e^{2u_{k}} g_{\euc} \in \mathcal{M}(D)$ which satisfies \textnormal{(P1)-(P3)}.
		Assume that $\{u_{k}\}$ has no bubble. Then 
		\begin{align}
		\label{identity.disk}
		\lim_{r \rightarrow 0} \lim_{k \rightarrow \infty} \Area( D_{r}, g_{k})=0. 
		\end{align}
	\end{pro}
	
	\begin{proof}
		\textbf{Step 1:} 	
		We first show that (\ref{identity.disk}) holds if $\lambda_{k}=0$ for all $k$.  
		
		By  (P2), 
		$|\K_{g_{k}}|(D) \leq \Lambda$, then we may  assume $|\K_{g_{k}}|$ converges to a Radon measure $\nu$.  
		Then there exists some positive integer $m$ such that $\frac{(m-1)}{2} \epsilon_{0} \leq \nu(\{0\})<\frac{m}{2}\epsilon_0$,
		where $\epsilon_{0}$ is as in Theorem \ref{thm convergence of solutions of Radon measure when measure is small}. Select $\delta$ such that $\nu(D_\delta\backslash\{0\})<\frac{\epsilon_0}{4}$. 
		We shall prove the result by induction on $m$.
		
		First, we assume $ \nu(\{0\})  \leq \frac{\epsilon_{0}}{2}$ ($m=1$). 	  
		Since $ \nu( D_{\delta }  ) < \epsilon_{0}$ and $|\mathbb{K}_{g_{k}}|(D_{{1}/{2}\delta}) \leq \epsilon_{0}$ for sufficiently large $k$, 
		by Theorem \ref{thm convergence of solutions of Radon measure when measure is small}, $e^{2u_k}$ converges in $L^1$, which implies \eqref{identity.disk}.

		Now, we assume the result holds if $\frac{(m-1)}{2} \epsilon_{0} \leq \nu(\{0\}) \leq \frac{m}{2}\epsilon_0 (m \geq 1)$ and prove the case under the assumption $\frac{m}{2} \epsilon_{0} \leq \nu(\{0\}) \leq \frac{(m+1)}{2}\epsilon_0$. 
		For any $ x \in \overline{D_{{\delta}/{2}}}$, we set 
		\begin{align*}
		r_{k}(x)=\inf\left\{t:  |\mathbb{K}_{g_{k}}|( D_{t}(x))  \geq  \frac{\epsilon_{0}}{2} \right\},\s \textnormal{and}\s r_k=\inf_{\overline{D_{\delta/2}}}r_k(x). 
		\end{align*}

		For any sufficiently small $\epsilon>0$ and sufficiently large $k$, 
		\begin{align*}
		\frac{\epsilon_{0}}{2} \leq \nu(\{0\}) \leq |\mathbb{K}_{g_{k}}|(D_{\epsilon}   ),
		\end{align*}
		then $r_k\leq\epsilon$, which implies that $r_{k} \rightarrow 0$. By the lower semi-continuity of $r_{k}$, we can find $x_{k}$ such that
		$ r_{k}(x_{k})=r_{k}$ and $x_{k } \rightarrow 0$ while noting that $x_k\rightarrow x'\neq 0$ would imply 
		$\nu(\{x'\})\geq\frac{\epsilon_0}{2}$ so $\nu(D_\delta\backslash\{x'\}) \geq \frac{\epsilon_0}{2}$, which contradicts the fact that $\nu(D_\delta\backslash\{0\})<\frac{\epsilon_0}{4}$. 
		We obtain that for any $z \in \mathbb{R}^{2}$ and for sufficiently large $k$,
		\begin{align*}
		|\mathbb{K}_{g_{k}}|(  D_{r_{k}  }(r_{k}z+x_{k})) \leq \frac{ \epsilon_{0}}{2} .
		\end{align*}
		Since $D_{\frac{r}{2}} \subset D_{r}(x_{k})$ for any fixed $r$ when $k$ is large, it suffices to show 
		\begin{align*}
		\lim_{r \rightarrow 0} \lim_{k \rightarrow \infty} \Area(  D_{r}(x_{k}) ,g_{k} )=0. 
		\end{align*}
		
		We define
		\begin{align*}
		t_{k}&=\inf\left\{ t:|\mathbb{K}_{g_{k}}|( D_{t}(x_{k})) \geq \frac{m \epsilon_{0}}{2} \right\},
		\end{align*}
		and set \begin{align*}
		U_{k}(x)=u_{k}(t_{k}x+x_{k})+\log t_{k}, \quad
		g_{k}^{\prime } =e^{2 {U}_{k}} g_{\euc}.
		\end{align*}
		Since $\nu(\{0\})\geq\frac{m}{2}\epsilon_0$, we see $t_k\rightarrow 0$.

		Assume $|\K_{g_k'}|$ converges to $\nu'$ weakly. Put
		$$
		\S'=\left\{ y : \nu'( \{ y \} ) \geq \frac{\epsilon_{0}}{2} \right\}.
		$$ 
		
		For any fixed $r>0$ and sufficiently large $k$, 
		we divide $\Area(  D_{r}(x_{k}) ,g_{k} )$ into three parts as follows:
		\begin{align}
		\int_{D_{r}(x_{k})} e^{2u_{k}}=&\int_{ D_{r}(x_{k})   \backslash D_{ \frac{t_{k}}{r}  }(x_{k}) } e^{2u_{k}} +\int_{  D_{ \frac{t_{k}}{r}  }(x_{k}) } e^{2u_{k}} =\int_{ D_{r}(x_{k})   \backslash D_{ \frac{t_{k}}{r}  }(x_{k}) } e^{2u_{k}} +\int_{  D_{ \frac{1}{r}  }} e^{2U_{k}} \nonumber \\
		&=\int_{ D_{r}(x_{k})   \backslash D_{ \frac{t_{k}}{r}  }(x_{k}) } e^{2u_{k}} +\int_{  D_{ \frac{1}{r} } \backslash \cup_{y \in\S'}D_r(y) } e^{2U_{k}} +\int_{  \cup_{y \in\S'}D_r(y) } e^{2U_{k}} , \label{integral plus}
		\end{align}
		we will show that each part of (\ref{integral plus}) will tend to zero as $k \rightarrow \infty$ and then $r \rightarrow 0$. 
		
		Firstly, we show: for any $r>0$, 
		\begin{align}
		\label{integral estimate 1}
		\lim_{k\rightarrow\infty}\int_{D_{\frac{1}{r}} \backslash \cup_{y \in\S'}D_r(y)}e^{2U_k}=0.
		\end{align}
		
		By \cite[Lemma 2.5]{C-L2} and (P3),  
		for any fixed $R>0$ and $z \in 
		\mathbb{R}^{2}$, for any $q \in [1,2)$ and sufficiently large $k$, 
		\begin{align*}
		\int_{D_{R}} |\nabla {U}_{k}(x)|^{q} dx&=\int_{ D_{R t_{k}}(x_{k})  } t_{k}^{q-2} |\nabla u_{k}(z)|^{q} dz =R^{2-q} \cdot (R t_{k})^{q-2} \int_{ D_{R t_{k}}(x_{k}) } |\nabla u_{k}|^{q} \\
		&\leq C(R,q,\Lambda_{1},\Lambda_{2}).
		\end{align*}
		Let ${C}_{k}=\frac{1}{|D|   } \int_{D} {U}_{k}$. By  the Poincaré inequality,
		\begin{align*}
		\|{U}_{k}-{C}_{k}\|_{ W^{1,q}( D_{R} )  } \leq C(R,q,\Lambda_{1},\Lambda_{2}).
		\end{align*}
		Then we may assume ${U}_{k} -{C}_{k}\to {U}_{\infty}$ in $W^{1,q}_{\loc}(\mathbb{R}^{2})$. Since $\int_{D}e^{2U_k}=\int_{D_{t_k}(x_k)}e^{2u_k}$, by Jensen's inequality  and (P2), $C_k$ is bounded above.
		If $C_k$ is bounded below, by Theorem \ref{thm convergence of solutions of Radon measure when measure is small}, $U_k$ converges to a bubble. Since $\{u_k\}$ has no bubble, we must have $C_k\rightarrow-\infty$, then $U_k$ converges to $-\infty$ for a.e. $x$. By Theorem \ref{thm convergence of solutions of Radon measure when measure is small} again, we obtain
		\begin{align*}
		\lim_{k\rightarrow\infty}\int_{D_{\frac{1}{r}} \backslash \cup_{y \in\S'}D_r(y)}e^{2U_k}=0.
		\end{align*}
		
		Secondly, we show: 
		\begin{align}
		\label{integral estimate 2}
		\lim_{r\rightarrow 0}\lim_{k\rightarrow\infty} \int_{  \cup_{y \in\S'}D_r(y) } e^{2U_{k}} =0.
		\end{align}
		If $\displaystyle \frac{t_{k}}{r_{k}} \rightarrow (2\alpha)^{-1} <\infty$, then for any $z \in \mathbb{R}^{2}$, 
		\begin{align*}
		|\mathbb{K}_{g_k'} |( D_{  \frac{r_{k}}{t_{k}}   }(  z)  )=|\mathbb{K}_{g_{k}}|( D_{r_{k}} (t_{k}z+x_{k})) \leq \frac{\epsilon_{0}}{2}, 
		\end{align*}
		which implies that for sufficiently large $k$, $\displaystyle  |\mathbb{K}_{g_k'} | (D_{\alpha}(z)) \leq \frac{\epsilon_{0}}{2}$. Then 
		\begin{align*}
		\lim_{r \rightarrow 0} \lim_{k \rightarrow \infty}   \int_{D_r(y)}e^{2U_k}=0. 
		\end{align*}
		If $\displaystyle \frac{t_{k}}{r_{k}} \rightarrow \infty$, then by the choice of $r_{k}$ and $t_{k}$, for some small $\displaystyle \epsilon^{\prime \prime}<\frac{\epsilon_{0}}{100}$, for any $R>0$ and sufficiently large $k$, 
		\begin{align*}
		|\mathbb{K}_{g_k'} |( \overline{ D_{  \frac{r_{k}}{t_{k}} }})&=|\mathbb{K}_{g_{k}}|( \overline{D_{r_{k}}(x_{k}) } ) \geq \frac{\epsilon_{0}}{2},\\
		|\mathbb{K}_{g_k'} |(D)&= |\mathbb{K}_{g_{k}}|( D_{t_{k}}(x_{k}) ) \in \left[\frac{m \epsilon_{0}}{2}, \frac{m \epsilon_{0}}{2}+\epsilon^{\prime \prime}\right],  \\
		|\mathbb{K}_{g_k'} |(  D_R   )&= |\mathbb{K}_{g_{k}}|( D_{Rt_{k}}(x_{k}) ) \leq |\mathbb{K}_{g_{k}}|(D_{\frac{1}{2} \delta}) \leq \frac{(m+1)\epsilon_{0} }{2}+\frac{\epsilon_{0}}{4}+\epsilon{''}.
		\end{align*}
		Then we have
		\begin{align*}
		\nu'(\{0\}) \geq \frac{\epsilon_{0}}{2}, \quad \nu'(D) \leq \frac{m \epsilon_{0}}{2}+\epsilon^{\prime \prime},\quad \nu'(\mathbb{R}^{2} \backslash D) <\epsilon_{0}
		\end{align*}
		which yield
		\begin{align*}
		\nu'(\{y\}) \leq \max \left\{\frac{(m-1) \epsilon_{0}}{2}+\epsilon^{\prime \prime},\epsilon_{0} \right\}<\frac{m\epsilon_{0}}{2}, \quad \forall y \in \S'. 
		\end{align*}
		Now it follows from the induction hypothesis that 
		$$
		\lim_{r\rightarrow 0}\lim_{k\rightarrow\infty}\int_{D_r(y)}e^{2U_k}=0. \quad \forall y \in \S'. 
		$$
		
		Thirdly, we show: 
		\begin{align}
		\label{integral estimate 3}
		\lim_{r\rightarrow 0}\lim_{k\rightarrow\infty}\Area(D_r(x_k)\backslash D_{\frac{t_k}{r}}(x_k),g_k)=0.
		\end{align}
		By (P1) and Lemma \ref{lem area identity on on neck region}, we only need to  show 
		$$
		\lim_{r\rightarrow 0}\lim_{k\rightarrow 0}\sup_{t\in[\frac{t_k}{r},r]}\Area(D_t\backslash D_{ \frac{t}{e} }(x_k),g_k)=0.
		$$
		Assuming this is not true, we can find $t_k'$, such that $t_k'\rightarrow 0$, $\frac{t_k'}{t_k} \rightarrow\infty$ and
		$$
		\Area(D_{t_k'}\backslash D_{{t_k'}/{e}}(x_k),g_k)\rightarrow\lambda>0.
		$$
		Let $\tilde{U}_k=u_k(x_k+t_k'x)+\log t_k'$. Then $\int_{D \backslash D_{{1}/{e}}}e^{2\tilde{U}_k}\rightarrow\lambda>0$. By Theorem \ref{thm convergence of solutions of Radon measure when measure is small},
		$\tilde{U}_k$ converges to a bubble, but this contradicts our assumption.
		
		Combining (\ref{integral plus}), (\ref{integral estimate 1}), (\ref{integral estimate 2}) and (\ref{integral estimate 3}), 
		we conclude that \eqref{identity.disk} holds if $\frac{m}{2} \epsilon_{0} \leq \nu(\{0\}) \leq \frac{(m+1)}{2}\epsilon_0$. Now the induction is complete. Hence (\ref{identity.disk}) holds if $\lambda_{k}=0$ for all $k$.     
		
		\textbf{Step 2:} 	
		We show the validity of (\ref{identity.disk}) without assuming $\lambda_{k}=0$ for all $k$.  
		
		By (P2), $\Area(D,g_{k})=\int_{D} e^{2u_{k}} \leq \Lambda_{1}$, then for any positive integer $N$, there exists a sequence $\tilde{r}_{k}^{N} \rightarrow 0$ such that
		\begin{align*}
		\Area(D_{  N\tilde{r}_{k}^{N} } ,g_{k} ) \leq \frac{1}{k}. 
		\end{align*}
		By a diagonal argument, we choose a sequence $\{ r_{k} \}$ by $r_{k}=\tilde{r}_{k}^{k}$, then $r_{k} \rightarrow 0$ and for fixed $r>0$, for any $k>\frac{1}{r}$, 
		\begin{align*}
		\Area( D_{ \frac{r_{k}}{r}} , g_{k}  ) \leq \Area(D_{  k r_{k} } ,g_{k} )  \leq \frac{1}{k}, 
		\end{align*}
		which yields that
		\begin{align}
		\label{bubble region estimate}
		\lim_{r \rightarrow 0} \lim_{k \rightarrow +\infty} \Area( D_{ \frac{r_{k}}{r}} , g_{k}  )=0.
		\end{align}
		
		Since
		\begin{align*}
		\Area( D_{r}, g_{k}   )= \Area( D_{r} \backslash D_{  \frac{r_{k}}{r} } , g_{k}   )+ \Area( D_{  \frac{r_{k}}{r} }, g_{k} ) , 
		\end{align*}
		it suffices to show
		\begin{align}
		\label{neck region estmate}
		\lim_{r \rightarrow 0} \lim_{k \rightarrow \infty} \Area( D_{r}\backslash D_{  \frac{r_{k}}{r} } , g_{k}   )=0. 
		\end{align}
		
		For any large $i\in\mathbb N$, we set $ u_{k}^{i}=u_{k}|_{D_{ 2^{-i+2}} \backslash  D_{ 2^{-i-2}}  }$ and define
		\begin{align*}
		A(k)=\left\{i: |\mathbb{K}_{g_{k}}|(  D_{ 2^{-i+2}} \backslash  D_{ 2^{-i-2}}) \geq \frac{\epsilon_{0}}{2} \right\}, \, B(k)=\left\{  i: |\mathbb{K}_{g_{k}}|(  D_{ 2^{-i+2}} \backslash  D_{ 2^{-i-2}}    ) < \frac{\epsilon_{0}}{2} \right \}
		\end{align*}
		where $\epsilon_{0}$ is as in Theorem \ref{thm convergence of solutions of Radon measure when measure is small}. The cardinality $\#A(k)$ has an upper bound depending only on $\Lambda_{1}$ and $\Lambda_{2}$. 
		
		Now for any fixed $i$, if $i \in B(k)$ for all large $k$, then by Theorem \ref{thm convergence of solutions of Radon measure when measure is small}, after passing to a subsequence, either $e^{2u_{k}^{i}}\to e^{2u^{i}}$ in $L^{1}( D_{ 2^{-i+1}} \backslash  D_{ 2^{-i-1}}   )$ for some $u^{i}$, or $e^{2u_{k}^{i}}\to 0$ in $L^{1}( D_{ 2^{-i+1}} \backslash  D_{ 2^{-i-1}}   )$. If not, after passing to a subsequence, we may assume for any $k$, $i \in A(k)$. We assume $|\mathbb{K}_{g_{k}}|\rightharpoonup\nu_{i}$ on $ D_{ 2^{-i+2}} \backslash  D_{ 2^{-i-2}}$, and set
		\begin{align*}
		\mathcal{S}_{i}=\left\{ y \in D_{ 2^{-i+2}} \backslash  D_{ 2^{-i-2}}: \nu_{i}( \{y\} )  \geq \frac{\epsilon_{0}}{2} \right\}.
		\end{align*}
		
		By (P1), $\mathbb{K}_{g_{k}}=f_{k} dV_{g_{k}}$ on  $D_{ 2^{-i+2}} \backslash  D_{ 2^{-i-2}}$, then by \textbf{Step 1}, 
		for any $y \in \mathcal{S}_{i}$,
		\begin{align*}
		\lim_{r \rightarrow 0} \lim_{k \rightarrow \infty} \int_{D_{r}(y)} e^{2u_{k}^{i}}=0. 
		\end{align*}
		Then combining with Theorem \ref{thm convergence of solutions of Radon measure when measure is small}, after passing to a subsequence, we still have: either $e^{2u_{k}^{i}}\to e^{2u^{i}}$ in $L^{1}( D_{ 2^{-i+1}} \backslash  D_{ 2^{-i-1}}   )$ for some function $u^{i}$, or $e^{2u_{k}^{i}}\to 0$ in $L^{1}( D_{ 2^{-i+1}} \backslash  D_{ 2^{-i-1}}   )$.
		Therefore, 
		\begin{align*}
		\lim_{i \rightarrow \infty} \lim_{k \rightarrow \infty} \Area( D_{ 2^{-i+1}} \backslash  D_{ 2^{-i-1}} ,g_{k}  )=0, 
		\end{align*}
		and (\ref{neck region estmate}) follows immediately. 
		
		Combining (\ref{bubble region estimate}) and (\ref{neck region estmate}), we see
		\begin{align*}
		\lim_{r \rightarrow 0} \lim_{k \rightarrow \infty} \Area( D_{r}, g_{k}   )= 
		\lim_{r \rightarrow 0} \lim_{k \rightarrow \infty}
		\Area( D_{r} \backslash D_{  \frac{r_{k}}{r} } , g_{k}   )+ 
		\lim_{r \rightarrow 0} \lim_{k \rightarrow \infty} \Area( D_{  \frac{r_{k}}{r} }, g_{k} )=0 .
		\end{align*}
		We complete the proof. 
	\end{proof}

	With the preparations above, we are ready to prove Theorem \ref{theorem area convergence of bubble tree}.     
	
	\begin{proof}
		\textbf{Step 1:} 
		We first show that (\ref{bubble tree area convergence}) holds if $\lambda_{k}=0$ for any $k$. 
		We shall argue by induction on $m$ = the number of bubbles. 
		If $m=0$, then $\{u_{k}\}$ has no bubbles. Applying Theorem \ref{thm convergence of solutions of Radon measure when measure is small} and Proposition \ref{pro no area concentration at zero on D}, there exists a finite set $\mathcal{S}$ such that
		\begin{align*}
		\lim_{k \rightarrow \infty} \Area(D_{{1}/{2}},g_{k})&= \lim_{r \rightarrow 0} \lim_{k \rightarrow \infty} \Area(D_{{1}/{2}} \backslash \cup_{z \in \mathcal{S}} D_{r}(z), g_{k})\\
		&+\sum_{z \in \mathcal{S}} \lim_{r \rightarrow 0} \lim_{k \rightarrow \infty} \Area(D_{r}(z),g_{k}) \\
		&=\lim_{r \rightarrow 0} \lim_{k \rightarrow \infty} \int_{ D_{{1}/{2}} \backslash \bigcup_{z \in \mathcal{S}} D_{r}(z)} e^{2u_{k}}+0 \\
		&=\Area(D_{{1}/{2}},g_{\infty}). 
		\end{align*}
		
		Next, we assume the assertion holds if $\{u_{k}\}$ has $m(\geq 0)$ bubbles and we will show that it still holds if  $\{u_{k}\}$ has $m+1$ bubbles. 
		
		\noindent \textbf{Case 1:} 
		We first consider the case that there is only one blowup sequence at 0, say $\{ (x_{k}^{1},r_{k}^{1}) \}$, which converges to a bubble at 1-level at 0. 
		
		By Definition \ref{def bubble right on top}, all the other bubbles can be regarded as bubbles of $u_{k}^{1}(x)=u_{k}(x_{k}^{1}+r_{k}^{1} x)+ \log r_{k}^{1}$. Since $u_{k}^{1}$ has $m$ bubbles, the induction hypothesis implies: for sufficiently large $R$, 
		\begin{align*}
		\lim_{ k \rightarrow \infty} \int_{ D_{Rr_{k}^{1} }(x_{k}^{1})   } e^{2u_{k}}=\lim_{k \rightarrow \infty} \int_{D_{R}} e^{2u_{k}^{1}}=\int_{D_{R}} e^{2v_{1}}+\sum_{i=2}^{m+1} \Area(\mathbb{R}^{2},g^{i}). 
		\end{align*}
		We point out that when $m=0$ the term $\sum_{i=2}^{m+1} \Area(\mathbb{R}^{2},g^{i})$ vanishes. 
		
		Applying Theorem \ref{thm convergence of solutions of Radon measure when measure is small} and Proposition \ref{pro no area concentration at zero on D} (in particular the proof of (\ref{neck region estmate})), 
		\begin{align*}
		\lim_{k \rightarrow \infty}& \Area(D_{\frac{1}{2}},g_{k})=\lim_{R \rightarrow \infty} \lim_{k \rightarrow \infty} \Area(D_{\frac{1}{2}} \backslash D_{Rr_{k}^{1}}(x_{k}^{1})  ,g_{k})+\lim_{R \rightarrow \infty} \lim_{k \rightarrow \infty} \Area(D_{Rr_{k}^{1}}(x_{k}^{1})  ,g_{k}) \\
		&=\lim_{R \rightarrow \infty} \lim_{k \rightarrow \infty} \Area(D_{\frac{1}{2}} \backslash D_{\frac{1}{R}}(x_{k}^{1}) ,g_{k})+
		\lim_{R \rightarrow \infty} \lim_{k \rightarrow \infty} \Area( D_{\frac{1}{R}}(x_{k}^{1})\backslash D_{Rr_{k}^{1}}(x_{k}^{1}),g_{k}   ) \\
		&+\lim_{R \rightarrow \infty} \lim_{k \rightarrow \infty} \Area(D_{Rr_{k}^{1}}(x_{k}^{1})  ,g_{k}) \\
		&= \Area(D_{\frac{1}{2}},g_{\infty} )+0+\lim_{R \rightarrow \infty} \int_{D_{R}} e^{2v_{1}}+\sum_{i=2}^{m+1} \Area(\mathbb{R}^{2},g^{i})\\
		&=\Area(D_{\frac{1}{2}},g_{\infty} )+\sum_{i=1}^{m+1} \Area(\mathbb{R}^{2},g^{i}). 
		\end{align*} 
		
		\noindent  \textbf{Case 2:} 
		We consider the case that there are $m'>1$ essentially different blowup sequences $\{ (x_{k}^{1},r_{k}^{1}) \}, \cdots \{ (x_{k}^{m'},r_{k}^{m'})\}$ of $\{u_{k}\}$ at 1-level all at the point 0 (\textbf{Case 2} can happen when $u_{k}$ has at least two bubbles; when $u_{k}$ has only one bubble, the assertion holds by \textbf{Case 1}). 
		
		We claim that $\frac{ |x_{k}^{\alpha}-x_{k}^{\beta}|}{r_{k}^{\alpha}+r_{k}^{\beta}   } \rightarrow \infty$ for any $\alpha \neq \beta \in \{1,\cdots,m^{\prime}\}$. Indeed, if $\frac{ |x_{k}^{\alpha}-x_{k}^{\beta}|}{r_{k}^{\alpha}+r_{k}^{\beta}}$ is bounded, then by Definition \ref{def bubble different}, either $\frac{ r_{k}^{\alpha}}{r_{k}^{\beta} } \rightarrow 0$ or $\frac{ r_{k}^{\beta}}{r_{k}^{\alpha} } \rightarrow 0$. WLOG, we assume $\frac{ r_{k}^{\alpha}}{r_{k}^{\beta} } \rightarrow 0$. Then $\frac{ x_{k}^{\alpha}-x_{k}^{\beta}}{r_{k}^{\beta} }$ converges, so 
		$\{ (x_{k}^{\alpha},r_{k}^{\alpha}) \}<\{ (x_{k}^{\beta},r_{k}^{\beta}) \}$. This contradicts  the assumption that $\{ (x_{k}^{\alpha},r_{k}^{\alpha}) \}$ and $\{ (x_{k}^{\beta},r_{k}^{\beta})\}$ are both at 1-level.
		
		WLOG, we assume 
		\begin{align*}
		\tilde{r}_{k}=2 \max\left\{ |x_{k}^{\alpha}-x_{k}^{\beta}|:\alpha ,\beta \in \{1,\cdots,m^{\prime}\} ,\alpha \neq \beta\right\}
		=2|x_{k}^{1}-x_{k}^{2}|. 
		\end{align*}
		For any $\alpha \in \{1,\cdots,m^{\prime}\}$, 
		\begin{align*}
		\frac{ \tilde{r}_{k}}{ r_{k}^{\alpha} } \geq 2 \max\Big\{ \frac{ |x_{k}^{1}-x_{k}^{\alpha}|}{r_{k}^{\alpha}}, \frac{ |x_{k}^{2}-x_{k}^{\alpha}|}{r_{k}^{\alpha}} \Big\} \rightarrow \infty, \text{ as } k \rightarrow \infty. 
		\end{align*}
		We set $\displaystyle \tilde{u}_{k}(x)=u_{k}(  \tilde{r}_{k} x +x_{k}^{1})+\log \tilde{r}_{k}$. By direct calculations, 
		\begin{align*}
		u_{k}^{\alpha}(x)=u_{k}\left(r_{k}^{\alpha}x+x_{k}^{\alpha}\right)+\log r_{k}^{\alpha} 
		%=u_{k}(  \tilde{r}_{k}(   \frac{r_{k}^{\alpha}}{\tilde{r}_{k}}x+\frac{ x_{k}^{\alpha}-x_{k}^{1}}{\tilde{r}_{k}}        )      +x_{k}^{1}     ) +\log \frac{ r_{k}^{\alpha}}{\tilde{r}_{k}} +\log \tilde{r}_{k} \\
		= \tilde{u}_{k}\Big(\frac{r_{k}^{\alpha}}{\tilde{r}_{k}}x+\frac{ x_{k}^{\alpha}-x_{k}^{1}}{\tilde{r}_{k}} \Big)+\log \frac{ r_{k}^{\alpha}}{\tilde{r}_{k}} ,
		\end{align*}
		which implies that any bubble $v^{\alpha}$ (hence all the bubbles sitting over $v^{\alpha}$) is also a bubble of $\{ \tilde{u}_{k} \}$. We select $\tilde{c}_{k}$ such that $\tilde{u}_{k}-\tilde{c}_{k}$ converges weakly in $W_{\loc}^{1,p}(\mathbb{R}^{2})$, then $\tilde{c}_{k} \rightarrow -\infty$ by Jensen's inequality and the fact that $\{u_{k}\}$ has only $m+1$ bubbles. In other words, $\tilde{u}_{k}$ converges to a ghost bubble. 
		By the choice of $\tilde{r}_{k}$, $|\frac{ x_{k}^{\alpha}-x_{k}^{1}}{\tilde{r}_{k}} |=0$ if we take $\alpha=1$, while $|\frac{ x_{k}^{\alpha}-x_{k}^{1}}{\tilde{r}_{k}}  |=\frac{1}{2}$ if we take $\alpha=2$. Therefore, $\tilde{u}_{k}$ has at least 2 blowup points, and we can isolate them by taking two small disks. Therefore, at each of the blowup points, there are at most $m$ bubbles, and there is no bubble outside the blowup points, and the assertion follows from the induction assumption and the fact that $\tilde{c}_{k} \rightarrow -\infty$.  
		
		\textbf{Step 2:} Now we show that (\ref{bubble tree area convergence}) still holds without assuming $\lambda_{k}=0$ for any $k$.
		\begin{comment}
		{\color{red}
		For any fixed $R$, we can choose $\rho_{k} \rightarrow 0$ such that for sufficiently large $k$, 
		\begin{align*}
		\Area(D_{R \rho_{k} }, g_{k} ) \leq \frac{1}{k}. 
		\end{align*}
		A diagonal argument then leads to 
		\begin{align*}
		\lim_{R \rightarrow \infty} \lim_{k \rightarrow \infty} \Area(D_{ R \rho_{k} } , g_{k})=0, 
		\end{align*}}
		\end{comment}
		As in {Step 2} of the proof of Proposition \ref{pro no area concentration at zero on D}, we can choose  $\rho_{k} \rightarrow 0$ such that for any fixed $R>0$, for sufficiently large $k$, 
		\begin{align*}
		\Area(D_{R \rho_{k} }, g_{k} ) \leq \frac{1}{k}.
		\end{align*}
		which yields that 
		\begin{align*}
		\lim_{R \rightarrow +\infty} \lim_{k \rightarrow +\infty}	\Area(D_{R \rho_{k} }, g_{k} ) =0,
		\end{align*}
		Hence we can view $\{(0,\rho_{k}) \}$ as a blowup sequence converging to a ghost bubble. 
		For any blowup sequence $\{(x_{k}^{i},r_{k}^{i})\}$, we claim that 
		\begin{align*}
		\text{either} \ \frac{|x_{k}^{i}|}{  r_{k}^{i}+\rho_{k}  } \rightarrow \infty, \text{ or } \ \frac{\rho_{k}}{r_{k}^{i}} \rightarrow 0.
		\end{align*}
		Indeed, if not, there exists $C_{0}>0$ such that 
		\begin{align*}
		\frac{|x_{k}^{i}|}{  r_{k}^{i}+\rho_{k}  } \leq C_{0}, \quad \lim_{k \rightarrow \infty} \frac{r_{k}^{i} }{\rho_{k}} =0. 
		\end{align*}
		Then for any fixed $\tilde{R}>0$, when $k$ is sufficiently large, 
		\begin{align*}
		\tilde{R} r_{k}^{i} +|x_{k}^{i}| \leq \tilde{R} r_{k}^{i} +C_{0}( r_{k}^{i}+\rho_{k}    ) \leq (2C_{0}+1) \rho_{k},
		\end{align*}
		which implies $D_{ \tilde{R} r_{k}^{i} }(x_{k}^{i}) \subset D_{ (2C_{0}+1) \rho_{k} }$. This means that $\{x_{k}^{i},r_{k}^{i}\}$ will also converge to a ghost bubble, this is impossible. Similar arguments as in {\bf Step 1} then lead to 
		\begin{align*}
		\lim_{k \rightarrow \infty} \Area( D_{\frac{1}{2}},g_{k}    )&=\lim_{R \rightarrow \infty} \lim_{k \rightarrow\infty} \Area( D_{\frac{1}{2}} \backslash D_{R \rho_{k}},g_{k})+\lim_{R \rightarrow \infty} \lim_{k \rightarrow \infty} \Area(  D_{R \rho_{k}},g_{k})
		\\
		&=\Area(D_{\frac{1}{2}},g_{\infty})+\sum_{i=1}^{m}  \Area(\mathbb{R}^{2},g^{i}).
		\end{align*} 
	\end{proof}
	
	%\vspace{-.5cm}
	\section{Negative curvature with $\chi(\Sigma,\beta)<0$}
	This section is devoted to the proof of Theorem \ref{thm existence of metrics with cusps}. 
	
	\subsection{Existence}
	
	\begin{pro}
		\label{pro negative curvature implies no bubble}
		Let $(\Sigma,g_0)$ be a closed surface and $g_k=e^{2u_k}g_{0}\in\mathcal{M}(\Sigma,g_0)$  with
		$$
		\K_{g_k}=f_k e^{2u_{k}} dV_{g_{0}}-2\pi  \sum_{i=1}^{m}  \beta_i^{k} \delta_{p_i},
		$$
		where $f_{k} \leq 0$ and 
		there exists a constant $\Lambda>0$ such that if
		$$
		\Area(\Sigma,g_k)+|\K_{g_k}|(\Sigma)<\Lambda
		$$
		then $\{u_k\}$ has no bubble.
	\end{pro}
	
	\begin{proof}
		Assuming the contrary, there exists a bubble at $x_0$. Taking an isothermal coordinate system with $x_0=0$, we can find $r_{k} \rightarrow 0$ and $x_{k} \rightarrow 0$ such that $\tilde{u}_{k}(x)=u_{k}(r_{k}x+x_{k})+\log r_{k}$ converges weakly in $W_{\loc}^{1,p}(\mathbb{R}^{2})$ to a function $\tilde{u}$.
		
		Let $\mu$ denote the weak limit of $\K_{e^{2\tilde{u}_k}g_{\euc}}$ up to contracting  a subsequence, which is a signed Radon measure defined on $\R^2$.
		We refer $\infty$ as the south pole $S$. By Theorem \ref{thm convergence of solutions of Radon measure when measure is small} and Lemma \ref{lemma gradient estimate on Riemann surfaces} (Lemma 2.3 is used to guarantee the gradient assumptions in Theorem 2.2) for any $R>0$, for sufficiently large $k$, 
		\begin{align*}
		\int_{D_{R}} e^{2 \tilde{u}  } &=\lim_{r \rightarrow 0} \lim_{k \rightarrow\infty} \int_{D_{R} \backslash \cup_{y \in \mathcal{S}} D_{r}(y) }e^{ 2\tilde{u}_{k}}\\
		&\leq \liminf_{k \rightarrow\infty} \int_{D_{R}} e^{ 2\tilde{u}_{k}}=\liminf_{k\rightarrow\infty} \int_{D_{R r_{k}}(x_{k})} e^{2u_{k}} \leq \Lambda,
		\end{align*}
		where $\epsilon_{0}$ is as in Theorem \ref{thm convergence of solutions of Radon measure when measure is small} and 
		\begin{comment}
		\begin{align*}
		\mathcal{S}=\{ y: \lim_{r \rightarrow 0} \lim_{k \rightarrow\infty } |\K_{e^{2\tilde{u}_k} g_{\euc}}|(\overline{D_{r}(y) }) \geq \frac{\epsilon_{0}}{2} \}.
		\end{align*}
		\end{comment}
		$$\mathcal{S}=\big\{ y: \lim_{r \rightarrow 0}\liminf_{k\to\infty}|\K_{e^{2\tilde{u}_k} g_{\euc}}|(D_{r}(y)) \geq \frac{\epsilon_{0}}{2} \big\}. $$

		By Lemma \ref{lemma removable singularity},  $\tilde{g}=e^{2\tilde{u}}g_{\euc}$ can be considered as a metric in $\mathcal{M}(\mathbb{S}^2)$.

		\noindent \textbf{Case 1:} $x_0\notin\{p_1,\cdots,p_m\}$. 
		
		Then $\mu$ is non-positive. We extend $\mu$ to a measure over $\mathbb S^2$ by defining $\mu(\{S\})=0$.
		Since $\Area(\mathbb{S}^2, \tilde{g})<\infty$, 
		by Lemma \ref{lemma removable singularity}, there exists $\lambda \geq -1$ such that 
		$$
		\K_{\tilde{g}}=\mu-2\pi\lambda\delta_{S}.
		$$
		Therefore, 
		\begin{align*}
		-\Delta_{\mathbb{S}^{2}} \tilde{u}=\mu-1-2 \pi \lambda \delta_{S},
		\end{align*}
		which implies that
		\begin{align*}
		0=\mu(\mathbb{S}^{2})-4 \pi -2 \pi \lambda \leq -2\pi,
		\end{align*}
		a contradiction. 
		
		\noindent \textbf{Case 2:} $x_0=p_1$, and $\frac{|x_k|}{r_k}\rightarrow\infty$. 
		
		For this case we can also conclude non-positivity of  $\mu$ and obtain a contradiction as in Case 1.
		
		\noindent \textbf{Case 3:} $x_0=p_1$ and $ \frac{ -x_{k}}{r_{k}} \rightarrow \tilde{x}$. 
		
		For this case, $\K_{e^{2\tilde{u}_k}g_{\euc}}$ is non-positive on $D_\frac{1}{r}\backslash D_r(\tilde{x})$. Then for any non-negative $\varphi \in \mathcal{C}_{c}^{\infty}(D_\frac{1}{r}\backslash D_r(\tilde{x}))$, we have 
		$$
		\int_{\R^2}\varphi d\mu=\lim_{k\rightarrow\infty}
		\int_{D_\frac{1}{r}\backslash D_r(\tilde{x})}\varphi d\K_{e^{2\tilde{u}_k}g_{\euc}}\leq 0,
		$$
		which implies that $\mu$ is non-positive on $\mathbb{S}^2\backslash\{S,\tilde{x}\}$. 
		
		Let $\nu=\mu\lfloor(\R^2\backslash\tilde{x})$.
		By Lemma \ref{lemma removable singularity}, there exist $\lambda_{1} \geq -1, \lambda_{2} \geq -1$ such  that
		$$
		\K_{\tilde{g}}=\nu-2\pi\lambda_1\delta_{\tilde{x}}-2\pi\lambda_2\delta_{S},
		$$
		where $\nu$ is non-positive. 
		
		Therefore, 
		\begin{align*}
		-\Delta_{\mathbb{S}^{2}} \tilde{u}=\nu-1-2 \pi \lambda_{1} \delta_{\tilde{x}}-2 \pi \lambda_{2} \delta_{S},
		\end{align*}
		which implies that
		\begin{align*}
		0=\nu(\mathbb{S}^{2})-4 \pi -2 \pi \lambda_{1}  -2\pi \lambda_{2} \leq 0.
		\end{align*}
		Then we obtain 
		$$
		\nu=0,\s and\s \lambda_1=\lambda_2=-1,
		$$
		which leads to a contradiction to Lemma \ref{lemma removable singularity}.
	\end{proof}
	
	The result below on the convergence of $\{u_k\}$ follows from Proposition \ref{pro no area concentration at zero on D} and Theorem \ref{thm convergence of solutions of Radon measure when measure is small}.

	\begin{pro}
		\label{prop convergence.with singular point}
		Let $g_k=e^{2u_k}g_{\euc}\in\mathcal{M}(D)$ with $\K_{g_k}=f_ke^{2u_k}dx-2\pi\beta_k\delta_0$. Assume 
		\begin{itemize}
			\item[(1)] $\Area(D,g_k)\leq\Lambda_1$;
			
			\item[(2)] \( r \|\nabla u_{k}\|_{L^{1}(D_r(x))} \leq \Lambda_{2} \), for all \( D_r(x) \subset D  \);
			
			\item[(3)] $-\Lambda_3 \leq f_k\leq -1$ or $1\leq f_k\leq \Lambda_3$,  and $f_k$ converges to $f$ for a.e. $x\in D$;
			
			\item[(4)] $\beta_k\rightarrow\beta$.
		\end{itemize}
		Assume $\{ u_k \}$ has no bubble. Then, after passing to a subsequence, one of the following holds:
		\begin{itemize}
			\item[(a)] $u_k\to u$ weakly in $\cap_{p \in [1,2)} W^{1,p}(D_\frac{1}{2})$
			and $e^{2u_k}\to e^{2u}$ in $L^1(D_\frac{1}{2})$, and 
			$$
			-\Delta u=fe^{2u}-2\pi \beta\delta_0;
			$$
			\item[(b)] $u_k\to -\infty$ a.e. and $e^{2u_k}\to 0$ in $L^1(D_\frac{1}{2})$.
		\end{itemize}
	\end{pro}
	
	\begin{proof}
		Let $c_k$ be the mean value of $u_k$ over $D_{{1}/{2}}$. By the Poincaré inequality and Sobolev inequality, we may assume $u_k-c_k$ converges weakly in $\cap_{p \in [1,2)}  W^{1,p}(D_{{1}/{2}})$. Let $\mu$ be the weak limit of $|\K_{g_k}|$, and
		$$
		\S=\left\{y:\mu(\{ y \})>\frac{\epsilon_0}{2}\right\}
		$$
		where $\epsilon_0$ is as in Theorem \ref{thm convergence of solutions of Radon measure when measure is small}. 
		
		Since $\Area(D,g_k)\leq\Lambda_1$, by Jensen's inequality, we may assume $c_k<C$.
		When $c_k\rightarrow-\infty$, by Theorem \ref{thm convergence of solutions of Radon measure when measure is small} and Proposition \ref{pro no area concentration at zero on D}, $e^{2u_k}\to 0$ in $L^1(D_{{1}/{2}})$. When $c_k$ converges, we may assume $u_k\to u$ weakly in $\cap_{p \in [1,2)}  W^{1,p}(D_{{1}/{2}})$. By Theorem \ref{thm convergence of solutions of Radon measure when measure is small}, $e^{2u_k}\to e^{2u}$ in $L^q(D_{{1/}{2}}\backslash \cup_{y \in\S}D_r(y))$ and $f_k\to f$ in $L^{q'}(D)$, where ${1}/{q}+{1}/{q'}=1$. Then
		\begin{align*}
		\lim_{k \rightarrow \infty} \int_{ D_{{1}/{2}} \backslash \cup_{y \in \S} D_{r}(y) } \varphi f_{k} e^{2u_{k}}=\int_{ D_{{1}/{2}} \backslash \cup_{y \in \S} D_{r}(y) } \varphi f e^{2u}.
		\end{align*}
		
		By Proposition \ref{pro no area concentration at zero on D}, 
		for any $y \in \S$, we have
		\begin{align*}
		&\lim_{r \rightarrow 0} \lim_{k \rightarrow \infty} \int_{D_{r}(y)} \varphi f_{k} e^{2u_{k}} \\
		=& \lim_{r \rightarrow 0} \lim_{k \rightarrow \infty} \int_{ D_{r}(y)} (  \varphi-\varphi(y)   ) f_{k} e^{2u_{k}}+\varphi(y) \lim_{r \rightarrow 0} \lim_{k \rightarrow \infty} \int_{D_{r}(y)} f_{k} e^{2u_{k}}=0.
		\end{align*}
		
		Therefore,
		\begin{align*}
		\int_{D_{{1}/{2}} } \nabla u \nabla \varphi=& \lim_{k \rightarrow \infty} \int_{D_{{1}/{2}} } \nabla u_{k} \nabla \varphi \\
		=& \lim_{k \rightarrow \infty} \Big(  \int_{D_{{1}/{2}} }  \varphi f_{k}  e^{2u_{k}} -2 \pi \beta_{k} \varphi(0) \Big) \\
		=& \int_{D_{{1}/{2}} }  \varphi f  e^{2u} -2 \pi \beta \varphi(0)  ,
		\end{align*}
		which implies
		\begin{align*}
		-\Delta u=f e^{2u}-2 \pi \beta \delta_{0}. 
		\end{align*}
	\end{proof}

	\noindent{\it Proof of Theorem \ref{thm existence of metrics with cusps}}. 
	As $K \in L^{\infty}(\Sigma)$, we can select $\Lambda>1$ and 
	$\phi_{k} \in C^{\infty}(\Sigma )$ with
	\begin{itemize}
		\item[(1)] $-\Lambda \leq \phi_{k} \leq -\frac{1}{\Lambda}$,
		\item[(2)] $\phi_{k} \rightarrow K$ in $L^{q}(\Sigma,g_{0})$ for any $q \geq 1$.
	\end{itemize}
	Since $K<0$ and $\Sigma$ is compact, by Theorem \ref{thm Troyanov negative curvature existence} we can find $u_{k}$ uniquely solving 
	\begin{align*}
	-\Delta_{g_{0}} u_{k} =\phi_{k} e^{2u_{k}}-K_{g_{0}}-2 \pi \sum_{i=1}^{m} \beta_{i}^{k} \delta_{p_{i}},
	\end{align*}
	where 
	\begin{align*}
	\beta_{i}^{k} \rightarrow \beta_{i}, \quad \beta_{i}^{k}>-1, \quad \chi(\Sigma,\beta^{k})=\chi(\Sigma)+\sum_{i=1}^{m} \beta_{i}^{k}<0.
	\end{align*}
	Set  $g_{k}=e^{2u_{k}} g_{0}$. Since 
	\begin{align*}
	\int_{\Sigma} \phi_{k} e^{2u_{k}} d V_{g_{0}}&=2 \pi \chi(\Sigma)+2 \pi \sum_{i=1}^{m} \beta_{i}^{k}=2 \pi \chi(\Sigma,\beta^{k}), \\
	\mathbb{K}_{g_{k}}&=\phi_{k} e^{2u_{k}} d V_{g_{0}}-2 \pi \sum_{i=1}^{m} \beta_{i}^{k} \delta_{p_{i}},
	\end{align*}
	we have
	\begin{align}
	\label{bound of area}
	0<-\frac{ 2 \pi \chi(\Sigma,\beta)}{\Lambda} \leq &\liminf_{k \rightarrow \infty} \Area(\Sigma,g_{k}) \leq \limsup_{k \rightarrow \infty} \Area(\Sigma,g_{k}) \leq -2 \pi \Lambda \chi(\Sigma,\beta)\\ \nonumber
	\textnormal{and} \hspace{1.1in}&\limsup_{k \rightarrow \infty} |\K_{g_k}|(\Sigma)\leq 2\pi \Lambda |\chi(\Sigma)|+4 \pi \sum_{i=1}^{m} |\beta_i|.
	\end{align}
	
	By Proposition \ref{pro negative curvature implies no bubble}, $\{ u_{k} \}$ has no bubble. 
	Let $c_k$ be the mean value of $u_k$ over $(\Sigma,g_{0})$. By Lemma \ref{lemma gradient estimate on Riemann surfaces}, after passing to a subsequence,
	we may assume $u_k-c_k$ converge weakly in $\cap_{p \in[1,2)} W^{1,p}(\Sigma,g_{0})$. 
	By (\ref{bound of area}) and Jensen's inequality, either $c_{k} \rightarrow c_{\infty}>-\infty$, or $c_{k} \rightarrow -\infty$. 
	If $c_k\rightarrow-\infty$, then $u_k\to-\infty$ a.e. Then we deduce from Proposition \ref{prop convergence.with singular point} that $\Area(\Sigma,g_{k})\rightarrow 0$, which 
	leads a contradiction to  (\ref{bound of area}).
	
	Therefore, we can assume $u_k\to u$ weakly in $\cap_{p \in[1,2)} W^{1,p}(\Sigma,g_{0})$. 
	By Proposition \ref{prop convergence.with singular point} again, $u$ is the desired solution.
	\endproof
	
	\subsection{Uniqueness}
	
	Now we show the solution in Theorem \ref{thm existence of metrics with cusps} is unique. 
	\begin{proof}
		\textbf{Step 1:} Let $u \in \cap_{p \in [1,2)} W^{1,p}(\Sigma,g_{0}) $ be a solution of the equation:
		\begin{align*}
		-\Delta_{g_{0}} u =K e^{2u}-K_{g_{0}}-2 \pi \sum_{i=1}^{m} \beta_{i} \delta_{p_{i}}. 
		\end{align*}
		For any $p_{i}$, we choose an isothermal coordinate system with $p_{i}=0$, then locally, 
		\begin{align*}
		-\Delta u=K e^{2u} -2 \pi \beta_{i} \delta_{0}. 
		\end{align*}
		We claim:
		\begin{align*}
		\lim_{r \rightarrow 0} r \| \nabla v \|_{ C^{0}(  \partial D_{r}  ) }=0,
		\end{align*}
		where $v(x)=u(x)-\beta_{i} \log |x|$. 
		
		Choosing  any sequence $\{ r_{k} \}$ with $r_{k} \rightarrow 0$, we set
		$$u_{k}(x)=u (r_{k} x)+\log r_{k}, \quad v_{k}(x)=v(r_{k}x).$$ For any $R>0$, on $D_{R} \backslash D_{ \frac{1}{R} }$, 
		\begin{align*}
		-\Delta u_{k}(x)=-r_{k}^{2} \Delta u( r_{k} x )=r_{k}^{2} K(r_{k} x) e^{ 2u(r_{k} x) }=K(r_{k} x) e^{ 2u_{k}(x)  }, 
		\end{align*}
		which yields that
		\begin{align*}
		\lim_{k \rightarrow +\infty} \Big| \int_{ D_{R} \backslash D_{ \frac{1}{R} }} K(r_{k} x ) e^{ 2u_{k}(x) } \Big| \leq \lim_{ k \rightarrow +\infty} \|K \|_{ L^{\infty}(\Sigma) } \int_{ D_{R r_{k}  \backslash D_{\frac{r_{k}}{R} } } }e^{2u} =0,
		\end{align*}
		where the last equality follows from $K \leq -1$ and 
		\begin{align}
		\label{finiteness of area}
		\int_{\Sigma} e^{2u} dV_{g_{0}}	 \leq \int_{\Sigma} -K e^{2u} dV_{g_{0}}=-\int_{\Sigma} K_{g_{0}} dV_{g_{0}}	 -2 \pi \sum_{i=1}^{m} \beta_{i}=-2 \pi \chi(\Sigma, \beta). 
		\end{align}
		It follows that the area of $\Sigma$ in $g=e^{2u}g_{0}$ is finite.  By Jensen's inequality,  the mean value $c_k'$ of $u_k$ over $D_R\backslash D_{1/R}$ is bounded from above. 
		By the Poincar\'e inequality, $\|u_k-c_k'\|_{L^1(D_R\backslash D_{1/R})}$ is bounded from above.
		Then by \cite[Corollary 2.4]{C-L2}, for any $p>1$, for sufficiently large $R$, there exists a constant $C=C(p,R)>0$ such that
		\begin{align*}
		\int_{ D_{R} \backslash D_{ \frac{1}{R} }} e^{2p(u_{k} -c_k')  } <C(p,R),
		\end{align*}
		which implies that
		\begin{align*}
		\int_{ D_{R} \backslash D_{ \frac{1}{R} }} e^{2pu_{k}   } <C(p,R).
		\end{align*}		
		By the definition of $v$, 
		\begin{align*}
		-\Delta v=Ke^{2u }(=K e^{2v} |x|^{2 \beta_{i}} ), \quad r^{-1} \| \nabla v \|_{ L^{1}(D_{r})  } \leq C. 
		\end{align*}
		Then
		\begin{align*}
		&	-\Delta v_{k}(x)=-r_{k}^{2} \Delta v_{k}(r_{k} x) =r_{k}^{2} K(r_{k}x)  e^{ 2u( r_{k}x )  }=K(r_{k}x) e^{ 2u_{k}(x)  }\\
		& R^{-1} \int_{D_{R}} |\nabla v_{k}|=(Rr_{k})^{-1} \int_{D_{R r_{k}}} |\nabla v| \leq C, 
		\end{align*}
		Now, by choosing $c_{k}$ to be the mean value of $v_{k}$ over $D_{2} \backslash D_{\frac{1}{2}}$, 
		$v_{k} -c_{k}$ converges to some function $v_{\infty}$ weakly in $W_{\loc}^{2,p}( \mathbb{R}^{2} \backslash \{0\} )$ with
		\begin{align*}
		&-\Delta v_{\infty}=0 \text{ on } \mathbb{R}^{2} \backslash \{0\}, \\
		& R^{-1} \int_{D_{R}} |\nabla v_{\infty} | \leq C, \quad \int_{ \partial D} \frac{\partial v_{\infty}}{\partial r}=0, 
		\end{align*}
		which yields that $v_{\infty}$ is a constant. Therefore, 
		\begin{align*}
		\lim_{k \rightarrow +\infty}	r_{k} \| \nabla v \|_{C^{0}( \partial D_{r_{k}}   )   }=\lim_{k \rightarrow +\infty} \| \nabla v_{k} \|_{ C^{0}( \partial D )  } =0. 
		\end{align*}
		By the arbitrariness of the sequence $\{r_{k}\}$, we obtain
		\begin{align*}
		\lim_{r \rightarrow 0} r \|\nabla v\|_{ C^{0}( \partial D_{r} )  }=0. 
		\end{align*}
		
		\textbf{Step 2:}
		Assume $u_{\alpha} \in \big(\cap_{p \in [1,2)} W^{1,p}(\Sigma,g_{0}) \big) \cap \big( \cap_{p>1} W_{\footnotesize\mbox{loc}}^{2,p}( \Sigma \backslash \{p_{1},\cdots,p_{m}\}, g_{0})\big)$ for $\alpha=1,2$ are two distinct solutions of the equation:
		\begin{align*}
		-\Delta_{g_{0}} u_{\alpha}=K e^{2u_{\alpha}}-K_{g_{0}}-2 \pi \sum_{i=1}^{m} \beta_{i} \delta_{p_{i}}. 
		\end{align*}
		We set $w=u_{1}-u_{2}$, then
		\begin{align*}
		-\Delta_{g_{0}} w=K e^{2u_{2}}( e^{2w}-1  ). 
		\end{align*}
		
		For any fixed $p_{i}$, by choosing an isothermal coordinate system with $p_{i}=0$ (the coordinate chart does not include $p_{j}$ for $j \neq i$), we set $v_{\alpha}(x)=u_{\alpha}(x)-\beta_{i} \log |x|$, then by \textbf{Step 1}, 
		$w=u_{1}-u_{2}=v_{1}-v_{2}$ satisfies
		\begin{align*}
		\lim_{ r \rightarrow 0} r \| \nabla w \|_{C^{0}(\partial D_{r})} \leq \lim_{ r \rightarrow 0} r (\| \nabla v_{1} \|_{C^{0}(\partial D_{r})}  +\| \nabla v_{2} \|_{C^{0}(\partial D_{r})}  )=0. 
		\end{align*}
		
		\textbf{Case 1}: we assume there exist $M>0$ and $|x_{k}|=t_{k} \rightarrow 0$ such that $|w(x_{k})|<M$. 
		
		For this case, 
		\begin{align*}
		\int_{D \backslash D_{t_{k}} }w \Delta w&=\int_{ D \backslash D_{t_{k}} } \big( \div( w \nabla w  )-|\nabla w |^{2} \big)\\
		&=\int_{\partial D} w \frac{\partial w}{\partial r}-\int_{\partial D_{t_{k}}} w \frac{\partial w}{\partial r}-\int_{D \backslash D_{t_{k}}} |\nabla w|^{2}.
		\end{align*}
		On the other hand, 
		\begin{align*}
		\int_{D \backslash D_{t_{k}} } w \Delta w=\int_{D \backslash D_{t_{k}} } K e^{2u_{2}}( 1-e^{2w} ) w \geq 0, 
		\end{align*}
		we obtain 
		\begin{align*}
		\int_{\partial D} w \frac{\partial w}{\partial r}-\int_{\partial D_{t_{k}}} w \frac{\partial w}{\partial r}-\int_{D \backslash D_{t_{k}}} |\nabla w|^{2} \geq 0. 
		\end{align*}
		For any $y \in \partial D_{t_{k}}$, 
		\begin{align*}
		|w(y)| \leq |w(x_{k})-w(y)|+|w(x_{k})| \leq \pi t_{k} \| \nabla w \|_{ C^{0}( \partial D_{t_{k}} ) }+M \leq 2M,
		\end{align*}
		then
		\begin{align*}
		\lim_{k \rightarrow \infty}\Big|\int_{ \partial D_{t_{k}}} w \frac{\partial w}{\partial r} \Big| \leq 2M \lim_{k \rightarrow \infty} 2 \pi t_{k} \| \nabla w\|_{ C^{0}(  \partial D_{t_{k}}  ) }=0,
		\end{align*}
		which yields 
		\begin{align*}
		\int_{\partial D} w \frac{\partial w}{\partial r}-\lim_{k \rightarrow \infty} \int_{D \backslash D_{t_{k}}} |\nabla w|^{2} \geq 0. 
		\end{align*}
		
		\textbf{Case 2:} we assume $\lim_{|x| \rightarrow 0} |w(x)|=+\infty$. 
		
		For this case, either $\lim_{|x| \rightarrow 0} w(x)=+\infty$ or $\lim_{|x| \rightarrow 0} w(x)=-\infty$ (otherwise, there exists $y_{k}$ with $|y_{k}|\rightarrow 0$ such that $w(y_{k})=0$). 
		
		\textbf{Case 2.1:}
		$\lim_{|x| \rightarrow 0} w(x)=+\infty$. Then there exists $r_{0}<\frac{1}{4}$ such that
		\begin{align*}
		w(x) \geq 0, \quad x \in D_{r_{0}}. 
		\end{align*}
		We choose $r$ sufficiently small such that
		\begin{align*}
		c_{r}:=\min_{ \partial D_{r} } w \geq \sup_{x \in D \backslash D_{r_{0}} } w(x). 
		\end{align*}
		We define the truncated function $w_{r}$ as 
		\begin{align*}
		w_{r}(x)=
		\left\{\begin{array}{rll}
		&c_{r}  & \text{when }  w(x) >c_{r}  \\
		&w(x) & \text{when } w(x) \leq c_{r} ,
		\end{array}\right.
		\end{align*}
		Clearly, $w_{r} \in W^{1,p}(D) \cap L^{\infty}(D)$ with
		$w=w_{r}$ on $D \backslash D_{ r_{0}}$. 
		We define the cutoff function $\eta_{\epsilon}$ as
		\begin{align*}
		\eta_{\epsilon}(x)=\left\{\begin{array}{rll}
		&1 & \text{when }  |x| \leq 1-\epsilon \\
		&1-\frac{1}{\epsilon }(|x|-1+\epsilon) & \text{when } 1-\epsilon<|x| \leq 1.
		\end{array}\right.
		\end{align*}
		We observe that on $D \backslash D_{r_{0}}$, 
		\begin{align*}
		K e^{2u_{2}} (e^{2w}-1) \eta_{\epsilon} w_{r}=K e^{2u_{2}} (e^{2w}-1) \eta_{\epsilon} w \leq 0,
		\end{align*}
		and on $D_{r_{0}}$, $w \geq 0$, $w_{r} \geq 0$, then
		\begin{align*}
		K e^{2u_{2}} (e^{2w}-1) \eta_{\epsilon} w_{r} \leq 0. 
		\end{align*}
		Then we obtain
		\begin{align*}
		0 \geq 	\int_{D} K e^{2u_{2}}( e^{2w}-1  ) \eta_{\epsilon} w_{r}&= 	\int_{D} \left\langle \nabla w, \nabla ( \eta_{\epsilon} w_{r} )\right \rangle \\
		& =\int_{D} \eta_{\epsilon} \left\langle \nabla w, \nabla w_{r} \right\rangle +\int_{D} w_{r} \left\langle \nabla w, \nabla \eta_{\epsilon} \right \rangle . 
		\end{align*}
		Since $\nabla \eta_{\epsilon}(x)=-\frac{1}{\epsilon} \frac{x}{|x|}$, then
		\begin{align*}
		\lim_{\epsilon \rightarrow 0} \int_{D} w_{r} \left\langle \nabla w, \nabla \eta_{\epsilon} \right \rangle=\lim_{\epsilon \rightarrow 0}	-\frac{1}{\epsilon}	 \int_{D \backslash D_{1-\epsilon}  } w \frac{\partial w}{\partial r}  =-\int_{\partial D} w \frac{\partial w}{\partial r},
		\end{align*}
		which yields that
		\begin{align*}
		0  \geq \int_{D}  \left\langle \nabla w, \nabla w_{r} \right\rangle -\int_{\partial D} w \frac{\partial w}{\partial r} =\int_{ D \backslash \{  x \in D: w(x) > c_{r}   \} } |\nabla w|^{2} -\int_{\partial D} w \frac{\partial w}{\partial r}, 
		\end{align*}
		where 
		\begin{align*}
		\{  x: w(x) > c_{r}     \} \subset D_{r_{0}} \subset D_{\frac{1}{4}}. 
		\end{align*}

		\textbf{Case 2.2:} 
		$\lim_{|x| \rightarrow 0} w(x)=-\infty$. Then there exists $r_{0}^{\prime}<\frac{1}{4}$ such that
		\begin{align*}
		w(x) \leq 0, \quad x \in D_{ r_{0}^{\prime} }. 
		\end{align*}
		We choose $r$ sufficiently small such that
		\begin{align*}
		c_{r}^{\prime}:=\max_{ \partial D_{r} } w \leq \inf_{x \in D \backslash D_{r_{0}^{\prime}} } w(x). 
		\end{align*}
		We define the truncated function $w_{r}$ as 
		\begin{align*}
		w_{r}(x)=
		\left\{\begin{array}{rll}
		&	c_{r}^{\prime}  & \text{when }  w(x) <	c_{r}^{\prime} \\
		&w(x) & \text{when } w(x) \geq 	c_{r}^{\prime},
		\end{array}\right.
		\end{align*}
		Clearly, $w_{r} \in W^{1,p}(D) \cap L^{\infty}(D)$ with
		$w=w_{r}$ on $D \backslash D_{r_{0}^{\prime} }$.
		Applying similar arguments as in \textbf{Case 2.1}, we can obtain
		\begin{align*}
		0  \geq \int_{D}  \left\langle \nabla w, \nabla w_{r} \right\rangle -\int_{\partial D} w \frac{\partial w}{\partial r} =\int_{ D \backslash \{  x \in D: w(x) <	c_{r}^{\prime}  \} } |\nabla w|^{2} -\int_{\partial D} w \frac{\partial w}{\partial r}, 
		\end{align*} 
		where 
		\begin{align*}
		\{  x \in D: w(x) < c_{r}^{\prime}  \}  \subset D_{r_{0}^{\prime}} \subset D_{\frac{1}{4}}. 
		\end{align*}
		
		Combining all the above cases, we conclude that we can always find a sequence  $\{A^{k}\}$ with $A^{k} \subset D_{\frac{1}{4}}$ and $\Area(A^{k}, g_{\mathbb{R}^{2}}) \rightarrow 0$ such that 
		\begin{align*}
		\int_{\partial D} w \frac{\partial w}{\partial r} -\lim_{k \rightarrow +\infty} \int_{ D \backslash A^{k}} |\nabla w|^{2} \geq 0. 
		\end{align*}
		
		\textbf{Step 3:} We choose $\{ \Omega_{i} \}_{i=1}^{m}$ such that
		\begin{align*}
		&p_{i} \in \Omega_{i}, \quad \partial \Omega \text{ is smooth}, \\ 
		& \Omega_{i} \cap \Omega_{j} =\emptyset, \quad \forall i \neq j. 
		\end{align*}
		Indeed, each $\Omega_{i}$ is chosen to be the domain $D$ if we choose an isothermal coordinate system with $p_{i}=0$. 
		Then
		\begin{align*}
		0 \leq 	\int_{ \Sigma \backslash \cup_{i=1}^{m} \Omega_{i} } 	 K e^{2u_{2}}( 1-e^{2w}  ) w&=\int_{\Sigma \backslash \cup_{i=1}^{m} \Omega_{i} } 	 w \Delta_{g_{0}} w 
		=\int_{\Sigma \backslash \cup_{i=1}^{m} \Omega_{i} }  	\div( w \nabla w )-|\nabla_{g_{0}} w|^{2}\\
		&=-\sum_{i=1}^{m} \int_{\partial \Omega_{i}} w \frac{\partial w}{\partial \vec{n}_{i} }-\int_{\Sigma \backslash \cup_{i=1}^{m} \Omega_{i} }  |\nabla_{g_{0}} w|^{2}, 
		\end{align*}
		where $\vec{n}_{i}$ is the outward-pointing unit normal vector of $\Omega_{i}$. 
		
		By \textbf{Step 2}, for any $i \in \{1,\cdots,m\}$, there exists a sequence $\{ A_{i}^{k} \}$ such that
		\begin{align*}
		A_{i}^{k} \subset\subset \Omega_{i}, \quad \lim_{k \rightarrow \infty} \Area( A_{i}^{k}, g_{0}  )=0,\quad \int_{\partial \Omega_{i}} w \frac{\partial w}{\partial \vec{n}_{i} } -\lim_{k \rightarrow +\infty} \int_{ \Omega_{i} \backslash A_{i}^{k} } |\nabla_{g_{0}} w|^{2} \geq 0.
		\end{align*} 
		Then we have
		\begin{align*}
		0 & \leq -\sum_{i=1}^{m} \int_{\partial \Omega_{i}} w \frac{\partial w}{\partial \vec{n}_{i} }-\int_{\Sigma \backslash \cup_{i=1}^{m} \Omega_{i} }  |\nabla_{g_{0}} w|^{2}+\sum_{i=1}^{m} \Big( \int_{\partial \Omega_{i}} w \frac{\partial w}{\partial \vec{n}_{i} } -\lim_{k \rightarrow +\infty} \int_{ \Omega_{i} \backslash A_{i}^{k} } |\nabla_{g_{0}} w|^{2} \Big) \\
		&=-\lim_{k \rightarrow +\infty} \int_{\Sigma \backslash \cup_{i=1}^{m} A_{i}^{k} } |\nabla_{g_{0}} w |^{2}, 
		\end{align*}
		which implies $|\nabla w| \equiv 0$, so $w$ is a constant. 
		We may assume $u_{1}=u_{2}+C$, where $C >0$. 
		Since
		\begin{align*}
		e^{2C} \int_{\Sigma} K e^{2u_{2}} dV_{g_{0}} = \int_{\Sigma} K e^{2u_{1}} dV_{g_{0}}=\int_{\Sigma} K_{g_{0}} dV_{g_{0}}+2 \pi \sum_{i=1}^{m} \beta_{i}=	\int_{\Sigma} K e^{2u_{2}} dV_{g_{0}}<0, 
		\end{align*}
		so $e^{2C}=1$, but this contradicts $C>0$. 	\end{proof}
	
	\subsection{Completeness of $g$ and finiteness of $\Area(\Sigma,g)$}
	
	Finiteness of area is proved in (\ref{finiteness of area}). We now argue that $g$ is complete. 
	In fact, by Lemma 4.7 in \cite{C-L2}, for large $m_0$,
	$$
	C_1\leq\frac{d_g(S_m,S_{m+1})}{\ell_g(L_m)}\leq C_2,\,\,\forall m>m_0
	$$
	where 
	$
	S_m=\partial D_{e^{-mL}}$ and $ L_m=[e^{-(m+1)L},e^{-mL}]\times\{0\}.
	$
	Let $x_0\in \partial D$ be any fixed point and let $x_n\in D_{e^{-i_nL}}\backslash D_{e^{-(i_n+1)L}}$ where $i_n\to +\infty$ as $n\to+\infty$. There is a curve $\gamma_n$ in $D\backslash D_{e^{-i_nL}}$ with endpoints in $\partial D$ and $\partial D_{e^{-i_nL}}$ such that $\ell_g(\gamma_n)-1\leq d_g(x_0,x_n)\leq\ell_g(\gamma_n)$. We have the following uniform behaviour as $x_n\to 0$: 
	
	{\bf Case 1}. $\sum^{\infty}_{n=1}\ell_g(L_{n})=+\infty$.
	Then $\lim_{n\to+\infty}d_g(x_0,x_n)=+\infty$ because
	\begin{align*}
	d_g(x_0,x_n)&\geq \ell_g(\gamma_n)-1\\
	&\geq \sum^{i_n}_{k=0}d_g(S_{k},S_{k+1})-1\\
	&\geq C_1\sum^{i_n}_{k=0}\ell_g(L_k)-1.
	\end{align*}
	
	{\bf Case 2}.  $\sum^{\infty}_{n=1}\ell_g(L_{n})=C<+\infty$. The line from $0$ to $x_n$ 
	intersects $\partial D_{e^{-(i_n+1)}L}$ at $x'_n$ and at $y_n\in\partial D$. Then 
	\begin{align*}
	d_g(x_0,x_n)&\leq d_g(x_n,y_n)+d_g(y_n,x_0) \\
	&\leq C_2\sum^{i_n}_{k=0}\ell_g(S_k,S_{k+1})+d_g(y_n,x_0)\\
	&\leq C_2C + \diam(D,g).
	\end{align*}
	
	We conclude completeness of $g$ from above discussion. 
	
	\begin{comment}
	\bigskip 
	
	\textit{On behalf of all authors, the corresponding author states that there is no conflict of interest. Data sharing not applicable to this article as no datasets were generated or analysed during the current study. }
	\end{comment}
	
	\bibliographystyle{plain}
	\bibliography{CM1}

\end{document}